\documentclass[a4paper,11pt]{amsart}
 \usepackage{amsmath, amssymb}
  \usepackage{amsmath, amssymb}
  \usepackage{amsopn}	 \usepackage{amstext}
  \usepackage{amsthm}	 \usepackage{amscd}
  \usepackage{amsgen}	 \usepackage{latexsym}
  \usepackage{enumerate}
  \usepackage{graphicx} \usepackage{color} \usepackage{cancel} \usepackage{ulem}
\newtheorem{thm}{Theorem}[section] 

\newtheorem{prop}[thm]{Proposition}
\newtheorem{lem}[thm]{Lemma}

\theoremstyle{definition}

\newcommand{\bpf}{\noindent{\it Proof. }}  	
\newcommand{\epf}{\hfill $\Box$}

 \newcommand{\R}{{\mathbf R}}  \newcommand{\C}{{\mathbf C}}

\newcommand{\bs}{{\boldsymbol{s}}} 
\newcommand{\bt}{{\boldsymbol{t}}}
\newcommand{\bu}{{\boldsymbol{u}}} 
\newcommand{\bv}{{\boldsymbol{v}}}

\newcommand{\bq}{{\boldsymbol{q}}} 
\newcommand{\br}{{\boldsymbol{r}}}

\def\GR{{\Gamma_{\R}}}
\def\GC{{\Gamma_{\C}}}
\def\GF{{\GR}}

\numberwithin{equation}{section}

\allowdisplaybreaks
\begin{document}

\title{Archimedean zeta integrals on $ {\rm SO}_{2n} \times {\rm GL}_n$ and $ {\rm Sp}_{2n} \times {\rm GL}_n $}
\author{Taku ISHII and Eric Stade}
\address{Faculty of Global Studies and Sustainability, Seikei University, 3-3-1 Kichijoji-Kitamachi,
Musashino, Tokyo, 180-8633, Japan}
\email{ishii@gss.seikei.ac.jp}
\address{Department of Mathematics, University of Colorado, Boulder, CO 80309, USA}
\email{eric.stade@colorado.edu}

\renewcommand{\thefootnote}{\fnsymbol{footnote}}

\date{\today}

\maketitle

\date{\today}
\maketitle 

\begin{abstract}
We evaluate the archimedean Rankin-Selberg zeta integrals associated with spherical principal series representations of 
${\rm SO}_{2n}(\R) \times {\rm GL}_n(\R) $ and 
$ {\rm Sp}_{2n}(\R) \times {\rm GL}_n(\R) $, where $ {\rm SO}_{2n} $ is split and $n\ge2$. 
We obtain explicit formulas expressing these integrals as quotients of products of gamma functions. 
Our approach uses Mellin-Barnes representations of spherical Whittaker functions, 
together with new integral identities for their Mellin transforms and transformations of Barnes integrals. 
In the even orthogonal case, we establish a recurrence in the rank and evaluate the integral by induction. 
The symplectic case then follows from an explicit identity relating the two families of zeta integrals.

\end{abstract}

\section*{Introduction}

Explicit evaluations of local zeta integrals are an important part of integral representations of automorphic $L$-functions. 
After unfolding, a global integral gives rise to local integrals involving Whittaker functions and, 
in many constructions, sections of induced representations.
At almost all finite places, the local data are unramified, and the corresponding calculation identifies the Euler factors represented by
the global integral. The archimedean factors require a different analysis,
in which explicit formulas for Whittaker functions play a central role.

The integrals studied in this paper arise from Methods B and C, respectively, of Gelbart and Piatetski-Shapiro in 
\cite[Part~B]{GePS},
for the pairs ${\rm SO}_{2n} \times {\rm GL}_n$ and ${\rm Sp}_{2n} \times {\rm GL}_n$. 
Here ${\rm SO}_{2n}$ is split and $n\ge2$. 
Their nonarchimedean unramified evaluation is described in \cite[Part~B, Appendix, \S A.2]{GePS}. 
More precisely, let $v$ be a finite place at which all the data are unramified, 
and let $\pi_v$ and $\pi'_v$ be the representations of
$G (F_v)$ and ${\rm GL}_n(F_v)$, respectively, where $G $ denotes either the split orthogonal group ${\rm SO}_{2n}$ (of rank $n$) or the symplectic group ${\rm Sp}_{2n}$ (of rank $n$). 
With the standard spherical data and Haar measures, the local integral has the form
\[
 Z_v^{\circ}(s)
 =\frac{L(s,\pi_v \times \pi_v')}{L(2s,\pi'_v,\operatorname{Sym}^2)}.
\]
The numerator is the local $L$-factor for the tensor product of the
standard representations of the dual groups. The notation
$Z_v^{\circ}(s)$ includes the unramified Schwartz function in the
symplectic case. Here the inducing parameter is written as
$|\det|^{s-1/2}$, with normalized parabolic induction, and the spherical
section is taken before the additional $L$-factor normalization.
The appearance of the symmetric-square factor in both cases is explained
in \cite[Part~B, \S\S 13.2--13.3]{GePS}; in the symplectic case, the
Eisenstein series is defined on the metaplectic cover.

The purpose of the present paper is to establish explicit real analogues
of these unramified evaluations for spherical principal series
representations. For general linear groups, evaluations of this kind were obtained by the
second author in \cite{St01}, \cite{St}.

The recursive formulas for ${\rm GL}_n(\R)$ Whittaker functions established by the authors in \cite{IS07}
express functions of rank $n$ in terms of those of rank $n-1$.
In \cite{IS13}, we used such formulas to compute archimedean zeta integrals
for general linear groups and split odd orthogonal groups. 
Here we continue this approach for $ {\rm SO}_{2n}(\R) \times {\rm GL}_n(\R) $ and
$ {\rm Sp}_{2n}(\R) \times {\rm GL}_n(\R) $, obtaining
the same quotient of $L$-factors as in the unramified nonarchimedean case.

We describe our results in terms of the radial parts of spherical Whittaker functions. 
For $\mu=(\mu_1,\ldots,\mu_n)$ and $\nu=(\nu_1,\ldots,\nu_n)$ in $\C^n$, 
let $\widetilde W^A_{n,\mu}$, $\widetilde W^D_{n,\nu}$, and
$\widetilde W^C_{n,\nu}$ denote the $\rho$-shifted spherical Whittaker
functions on ${\rm GL}_n(\R)$, ${\rm SO}_{2n}(\R)$, and ${\rm Sp}_{2n}(\R)$, respectively. 
See Section \ref{sec:Wh} for the precise definitions and normalizations. Write
\begin{align*}
 \alpha[a_1,\ldots,a_n]&=\operatorname{diag}(a_1,\ldots,a_n) \in {\rm GL}_n(\R),\\
 \beta[a_1,\ldots,a_n] 
 &=\operatorname{diag}(a_1,\ldots,a_n,a_n^{-1},\ldots,a_1^{-1}) \in {\rm SO}_{2n}(\R), {\rm Sp}_{2n}(\R).
 \end{align*}
With the local data and measures specified in Sections \ref{sec:Arch_D} and \ref{sec:Arch_C}, 
the two zeta integrals take the forms
\begin{align*}
 I^n_{\nu,\mu}(s)
 & =\int_{(\R_+)^n} \widetilde W^D_{n,\nu}\bigl(\beta[a_1,\ldots,a_n]\bigr)
 \\
 &\times  \widetilde W^A_{n,\mu} \bigl(\alpha[a_1,\ldots,a_{n-1},(a_n+a_n^{-1})^{-1}]\bigr)
 \cdot \Bigl(\frac{a_1\cdots a_{n-1}}{a_n+a_n^{-1}}\Bigr)^{\!s}
 \prod_{i=1}^n\frac{da_i}{a_i}
\end{align*}
and
\begin{align*}
 J^n_{\nu,\mu}(s)
 & =\int_{(\R_+)^n}
 \widetilde W^C_{n,\nu}\bigl(\beta[a_1,\ldots,a_n]\bigr)
 \widetilde W^A_{n,\mu}\bigl(\alpha[a_1,\ldots,a_n]\bigr)
 \cdot (a_1\cdots a_n)^s
 \exp(-2\pi a_n^2)\prod_{i=1}^n\frac{da_i}{a_i}.
\end{align*}
The integrals are initially considered in a domain of absolute convergence;
the resulting identities are understood elsewhere by meromorphic
continuation.
Here is our main result.

\begin{thm} $($Theorems \ref{thm:SOGL} and \ref{thm:SpGL}$)$
Put $\GR(s)=\pi^{-s/2}\Gamma(s/2)$. We have
\[
 I^n_{\nu,\mu}(s)
 =  \frac{\prod_{1 \le i,j \le n}  \GR(s+\nu_i+\mu_j)  \GR(s-\nu_i+\mu_j) }{ 
\prod_{1\le i\le j\le n}  \GR(2s+\mu_i+\mu_j)}
\]
and
\[
 J^n_{\nu,\mu}(s)
 = \frac{ \prod_{1 \le i \le n} \GR (s+\mu_i)
 \prod_{1 \le i, j \le n} 
   \GR (s+\nu_i+\mu_j) \GR(s-\nu_i+\mu_j)   }{  \prod_{1\le i\le j\le n}  \GR(2s+\mu_i+\mu_j) },
\]with  $\Gamma_\R$ as in \eqref{gamma-f}.
\end{thm}

For generic parameters, the numerators are the corresponding
archimedean tensor product $L$-factors. 
The common denominator is the symmetric-square factor
\[
 L(2s,\pi_\mu,\operatorname{Sym}^2)
 =\prod_{1\le i\le j\le n}\GR(2s+\mu_i+\mu_j)
\]
associated with the ${\rm GL}_n(\R) $ inducing representation.


The main part of the proof is the evaluation of the even orthogonal
integral $I^n_{\nu,\mu}(s)$, which we achieve as follows. Using recursive Mellin-Barnes formulas for the Whittaker
functions, we establish a recurrence expressing $I^n_{\nu,\mu}(s)$ 
as an explicit product of gamma factors times
$I^{n-1}_{\tilde\nu,\tilde\mu}(s)$, where
$\tilde\nu=(\nu_1,\ldots,\nu_{n-1})$ and
$\tilde\mu=(\mu_1,\ldots,\mu_{n-1})$; see Proposition \ref{prop:SOGL_induction}.

The initial case -- computation of $I^2_{\nu,\mu}(s)$ -- corresponds to the spherical triple-product
calculation of Ikeda in \cite{Ik99}. Here we give a direct proof of this case, using identities for Barnes integrals.
The induction then relies on two additional, new  integral identities for 
Mellin transforms of spherical Whittaker functions.  Both of these new identities are related to an earlier result (\cite[Proposition 2.6]{IS13}), which expresses the Mellin transform of $\widetilde W^A_{n,\mu}$ as an integral of another Mellin transform of $\widetilde W^A_{n,\mu}$, against some gamma functions.  See Proposition \ref{prop:Pieri_GLn} below.  The first of the new identities -- see Proposition \ref{prop:new_GLn} below -- employs Proposition \ref{prop:Pieri_GLn} to obtain a formula relating two different integrals of $\widetilde W^A_{n,\mu}$ against gamma functions.  The second of the new identities -- see Proposition \ref{prop:Pieri_SO2n} -- is an even-orthogonal analogue of Proposition \ref{prop:Pieri_GLn}.
Transformations of Barnes integrals allow us to reorganize the kernels
so that these identities can be applied.  As a result, we obtain Theorem 3.2, which is the desired result  for the case of $ {\rm SO}_{2n}(\R) \times {\rm GL}_n(\R) $.

The symplectic evaluation is then obtained by a reduction to the even orthogonal one.
 The Mellin-Barnes formula in Proposition \ref{prop:exp_Sp2n} below relates
$\widetilde W^C_{n,\nu}$ to $\widetilde W^D_{n,\nu}$.
Combining this formula with the formula for the Mellin transform of a Gaussian,
Barnes' first lemma, and Proposition \ref{prop:Pieri_GLn}, gives
\[
 J^n_{\nu,\mu}(s)  = \prod_{j=1}^n\GR(s+\mu_j)  \cdot  I^n_{\nu,\mu}(s).
\]
This relation is obtained directly from the integral representations; it does not require {\it a priori} evaluation of $I^n_{\nu,\mu}(s)$.  Substituting that evaluation -- that is, substituting Theorem 3.2 -- 
then yields Theorem \ref{thm:SpGL}, which is the desired result in the case of $ {\rm Sp}_{2n}(\R) \times {\rm GL}_n(\R) $.

This paper is organized as follows. Section \ref{sec:Barnes} collects the Barnes
integral identities used in the proofs. Section \ref{sec:Wh} fixes the Whittaker normalizations, 
recalls the required Mellin-Barnes representations, 
and proves the additional integral identities. Section \ref{sec:Arch_D} identifies
the even orthogonal local integral and evaluates it by induction on the rank. 
Section \ref{sec:Arch_C} specifies the symplectic local data, and evaluates the zeta integral in question by reducing it to the even orthogonal integral.

\section{Formulas for Barnes integrals}  \label{sec:Barnes}
In this section, we prepare some identities for Barnes integrals. 
In the following, $ a,b,c, \dotsc \in \C $, and the notation $\int_z $  denotes integration over a vertical line in the complex plane,
  indented if necessary, that separates the increasing and decreasing sequences of poles.
  
For $ F = \R $ or $\C$ and $ s \in \C $, we write 
\begin{align}
 \Gamma_F (s) = 
\begin{cases} \pi^{-\frac{s}{2}} \Gamma(\frac{s}{2}) & \text{if $F=\R$}, \\ 2(2\pi)^{-s} \Gamma(s)  & \text{if $F=\C$}. 
 \end{cases}\label{gamma-f}
\end{align}

\subsection{Barnes' first lemma} 

\begin{lem} \label{lem:1st} $($Barnes' first lemma \cite[(4.2.1.4)]{Sl}$)$
We have 
\begin{align*}
& \frac{1}{4\pi \sqrt{-1}} \int_z \GR(z+a) \GR(z+b) \GR(-z+c) \GR(-z+d) \,dz 
\\
&  = \frac{\GR(a+c)\GR(a+d) \GR(b+c) \GR(b+d)}{\GR(a+b+c+d)}. 
\end{align*}
\end{lem}

\subsection{The integral $I_1(a,b,c; d,e; f) $}

\begin{lem} \label{lem:int1}
Let $I_1(a,b,c; d,e;f)$ be the Barnes integral
\begin{align*}I_1(a,b,c; d,e;f)
&= \frac{1}{4\pi \sqrt{-1}} \int_z \frac{ \GR(z+a) \GR(z+b) \GR(z+c) 
    \GR(-z+d) \GR(-z+e)}{\GR (z+f) } \,dz. 
\end{align*}
{\rm (i)}
We have 
\begin{align*}
I_1(a,b,c; d,e;f)
& = \frac{ \GR(a+d) \GR(a+e)}{\GR(f-b) \GR(f-c)} \cdot 
     I_1( d,e,f-b-c ; b,c; a+d+e).
\end{align*}
{\rm (ii)}
We have 
\begin{align*}
I_1(a,b,c;  d,e; f)
& = \frac{ \GR(a+d) \GR(b+e) \GR(c+e)}{\GR(a+b+c+d+e-f)\GR(f-b) \GR(f-c)}  
\\
& \times I_1(a, f-b-e, f-c-e; b+c+d+e-f,e; f).
\end{align*}
{\rm (iii)} {\rm (}Barnes' second lemma \cite[(4.2.2.2)]{Sl}{\rm)}
If $ a+b+c+d+e= f $, then we have 
\begin{align*}
I_1(a,b,c; d,e;f)
& = \frac{\GR(a+d)\GR(a+e) \GR(b+d)\GR(b+e)
     \GR(c+d)\GR(c+e)}{\GR(a+b+d+e)  \GR(b+c+d+e) \GR(c+a+d+e)}.
\end{align*}
\end{lem}
\bpf
By Lemma \ref{lem:1st}, we have
\begin{align*}
 I_1(a,b,c; d,e;f)
& = \frac{1}{ (4\pi \sqrt{-1})^2} \int_{z,w} 
   \frac{ \GR(z+a) \GR(-z+d) \GR(-z+e) } { \GR(f-b) \GR(f-c)}  \\
&  \times 
      \GR(w+b) \GR(w+c) 
      \GR(-w+z) \GR(-w+f-b-c) \,dz \,dw.
\end{align*}
We use Lemma \ref{lem:1st} for the integration over $z$ to get (i).
To get (ii), we use (i) twice: 
\begin{align*}
 I_1(a,b,c; d,e;f)
& = \frac{ \GR(a+d) \GR(b+e) \GR(c+e) }{ \GR(a+b+c+d+e-f) \GR(f-b) \GR(f-c) }
\\
& \times 
  I_1(b,c,a+b+c+e-f; d, f-b-c; b+c+e).
\end{align*} 
Then the substitution $ z \to z+f-b-c-e $ implies (ii). 
\epf

\subsection{The integral $I_2(a,b,c; d,e; f) $}

\begin{lem} \label{lem:int2} 
Let $I_2(a,b,c; d,e;f)$ be the Barnes integral
\begin{align*}I_2(a,b,c; d,e;f)=
  \frac{1}{4\pi \sqrt{-1}} \int_z 
 \frac{ \GR(z+a) \GR(z+b) \GR(2z+c) \GR(-z+d) \GR(-z+e)} 
  {\GR(2z+f)} \,dz.
\end{align*}
If $ a+b +c + d +e =f $, then we have 
\begin{align*}
 I_2(a,b,c; d,e; f) 
& = \frac{ \GR(a+d) \GR(a+e)}{ \GR(f-c) \GR(c+d+e) } 
\cdot 
  I_1(b, c+d, c+e; d,e; f-b). 
\end{align*}
\end{lem} 

\bpf 
Lemma \ref{lem:1st} implies that
\begin{align*}
 \frac{ \GR(z+b) \GR(2z+c)}{ \GR(2z+f) } 
& =  \frac{1}{ \GR(f-c) \GR(z+f-b)} \cdot  
\frac{1}{4\pi \sqrt{-1}} \int_w \GR(z-w) 
\\
& \times  \GR(-w+f-b-c) \GR(w+b) \GR(z+w+c) \,dw.
\end{align*}
Then we have 
\begin{align*}
 I_2(a,b,c; d,e;f)
& = \frac{1}{\GR(f-c)} \cdot \frac{1}{(4\pi \sqrt{-1})^2} \int_{z,w} 
   \GR(z+a) \GR(-z+d) \GR(-z+e)  
\\
& \times \frac{  \GR(z-w) \GR(-w+f-b-c) \GR(w+b) \GR(z+w+c)} 
  {\GR(z+f-b)} \,dz \,dw.
\end{align*}
Since $ a+b +c + d +e =f $, we can use Barnes' second lemma for $ \int_z $ to get 
\begin{align*}
 I_2(a,b,c; d,e;f)
& = \frac{\GR(a+d) \GR(a+e)  }{\GR(f-c) \GR(c+d+e) } \cdot 
 \frac{1}{4\pi \sqrt{-1}} \int_{w} 
    \GR(-w+d) \GR(-w+e)  
\\
& \times \frac{  \GR( w+c+d) \GR( w+c+e) \GR(w+b)} 
  { \GR(w+f-b) }   \,dw,
\end{align*}
as desired.
\epf

\subsection{The integral $ I_3(a,b,c,d; \alpha,\beta,\gamma,\delta) $}

The next formula is due to the second author \cite{St}. 
See also \cite[Lemma 1.5]{I19} for an alternative proof by using Lemma \ref{lem:int1} (i).

\begin{lem} \label{lem:int3} $($\cite[Lemma 2.1]{St}$)$ 
Let 
$ I_3(a,b,c,d; \alpha,\beta,\gamma,\delta) $ be the Barnes integral
\begin{align*}&I_3(a,b,c,d; \alpha,\beta,\gamma,\delta)\\&=
 \frac{1}{4\pi \sqrt{-1}} \int_z \frac{\GR (z+a)\GR (z+b) \GR (z+c)
  \GR (-z+\alpha) \GR (-z+\beta) \GR (-z+\gamma)}{ \GR (z+d) \GR (-z+\delta) } \,dz. 
\end{align*}
If $ a+b +c + \alpha + \beta + \gamma = d +\delta $, then we have 
\begin{align*}
I_3(a,b,c,d; \alpha,\beta,\gamma,\delta) 
&  =\frac{ \GR(a+\gamma)\GR(b+\gamma) \GR(\alpha+c ) \GR(\beta+c) }{ 
 \GR(d-a)\GR(d-b) \GR(\delta-\alpha) \GR(\delta-\beta)} \\
& \times I_3(a,b,\delta-\alpha-\beta,a+b+\gamma; \alpha,\beta,d-a-b, \alpha+\beta+c).
\end{align*}  
\end{lem}

\subsection{The integral $ I_4(a,b,c,d; \alpha,\beta,\gamma,\delta) $}

\begin{lem} \label{lem:int4}
Let 
$ I_4(a,b,c,d; \alpha,\beta,\gamma,\delta) $ be the Barnes integral 
\begin{align*}&I_4(a,b,c,d; \alpha,\beta,\gamma,\delta)\\&=
 \frac{1}{4\pi \sqrt{-1}} \int_z \frac{\GR (z+a)\GR (z+b) \GR (2z+c)
  \GR (-z+\alpha) \GR (-z+\beta) \GR (-z+\gamma)}{ \GR (2z+d) \GR (-z+\delta) } \,dz. 
\end{align*}
If $ a+b +c + \alpha + \beta + \gamma = d +\delta $, then we have 
\begin{align*}
&  I_4(a,b,c,d; \alpha,\beta,\gamma,\delta) \\
& = \frac{\GR(a+\gamma) }{ \GF(d-c) \GR(d-a-b) \GR(\delta-\alpha) \GR(\delta-\beta) } 
\\
& \times   \frac{1}{(4\pi \sqrt{-1})^2} \int_{z,w} 
     \GF(w+b) \GR(w+c+\gamma)\GR(-w+\alpha ) \GR(-w+\beta) \GR(z+w+c)\\
& \times
 \frac{ \GR(z+a) \GR(z+\delta-\alpha-\beta) 
  \GR(-z+\alpha) \GR(-z+\beta) \GR(-z-w+d-a-b-c)  } 
 {\GR(z+w+a+c+\gamma) \GR(-z-w+\alpha+\beta)}\,dz\, dw.
\end{align*}
\end{lem}

\bpf
By Lemma \ref{lem:1st}, we know that
\begin{align*}
 \frac{ \GR(z+b) \GR(2z+c)}{ \GR(2z+d) } 
& =  \frac{1}{ \GR(d-c) \GR(z+d-b)} \cdot  
\frac{1}{4\pi \sqrt{-1}} \int_w \GR(z-w) 
\\
& \times  \GR(-w+d-b-c) \GR(w+b) \GR(z+w+c) \,dw.
\end{align*}
Then we have 
\begin{align*}
& \GR(d-c) \cdot I_4(a,b,c,d; \alpha,\beta,\gamma,\delta)  \\
& =  \frac{1}{(4\pi \sqrt{-1})^2} \int_{z,w} 
   \GR(z+a) \GR(-z+\alpha) \GR(-z+\beta) \GF(-z+\gamma)
\\
& \times \frac{  \GR(z-w) \GR(-w+d-b-c) \GR(w+b) \GR(z+w+c)} 
  {\GR(z+d-b) \GF(-z+\delta) } \,dz\, dw
\\
& = \frac{1}{4\pi \sqrt{-1}} \int_{w}
    I_3(w+c,a,-w, d-b; \alpha, \beta, \gamma, \delta)\GF(-w+d-b-c) \GF(w+b)  \,dw.
\end{align*}
Because of the assumption $ a+b +c + \alpha + \beta + \gamma = d +\delta $, we can use Lemma \ref{lem:int3}:
\begin{align*}
& I_3(w+c,a,-w, d-b; \alpha, \beta, \gamma, \delta)  \\
& = 
 \frac{ \GR(a+\gamma) }{ \GR(d-a-b) \GR(\delta-\alpha) \GR(\delta-\beta) }
\\
& \times 
\frac{1}{4\pi \sqrt{-1}} \int_{z} \frac{ \GR(w+c+\gamma)\GR(-w+\alpha ) \GR(-w+\beta) }{ 
 \GR(-w+d-b-c)} \cdot \GR(z+w+c) \GR(z+a) \\
& \times 
 \frac{\GR(z+\delta-\alpha-\beta) 
  \GR(-z+\alpha) \GR(-z+\beta) \GR(-z-w+d-a-b-c)  } 
 {\GR(z+w+a+c+\gamma) \GR(-z-w+\alpha+\beta)}\,dz,
\end{align*}
to get the assertion.
\epf

\section{Explicit formulas for Whittaker functions}  \label{sec:Wh}
Throughout this section, all multiple Barnes integrals are initially taken over vertical contours for which the increasing and decreasing sequences of poles are separated, and such that the integrals converge absolutely. The parameters are first restricted to a nonempty domain, on compact subsets of which the integrals converge uniformly, and  in which all changes in the order of integration and all changes of variables are justified. The resulting identities are then extended meromorphically to  parameters with arbitrary values (subject to any stipulated  relations among the parameters). Empty products are understood to be equal to 1.

\subsection{Spherical Whittaker functions} \label{subsec:classoneWh} 

Let $G $ be the group ${\rm GL}_n(\R)$, $ {\rm Sp}_{2n}(\R) $ or $ {\rm SO}_{2n} (\R) $.
Here 
${\rm Sp}_{2n}(\R) $ is the symplectic group with respect to 
$ \Bigl(\begin{smallmatrix} & J_n \\ -J_n & \end{smallmatrix}\Bigr)$, and
$ {\rm SO}_{2n}(\R)= {\rm SO}_{n,n}(\R)  $ is the split special orthogonal group with respect to 
$ \Bigl(\begin{smallmatrix} & J_n \\ J_n & \end{smallmatrix}\Bigr)$, 
where $ J_n $ is the anti-diagonal matrix $ \left( \begin{smallmatrix}
& & 1 \\ & \rotatebox[origin=c]{45}{$\cdots$} & \\ 1 & &
\end{smallmatrix} \right) $ of size $n$.

Fix a maximal compact subgroup $K$ of $G $ by 
\begin{align*}
K  = \begin{cases}
{\rm O}(n) & \mbox{ if } G  = {\rm GL}_n(\R), \\
G  \cap {\rm O}(2n) & \mbox{ if } G = {\rm Sp}_{2n}(\R), {\rm SO}_{2n}(\R),
\end{cases}
\end{align*}
where $ {\rm O}(n) =  \{ g \in {\rm  GL}_n(\R) \mid {}^tg g = 1 \} $.
Denote by $U$ the maximal unipotent subgroup of $G $ that consists of
upper triangular unipotent matrices in $G$.
Let $A $ be the subgroup of $G$ defined by 
\begin{align*}
A 
= \begin{cases}
 \{ \alpha[a_1,\dotsc, a_n] = {\rm diag}(a_1,\dotsc,a_n)   \mid a_i > 0 \} & \mbox{ if } G = {\rm GL}_n(\R), \\
 \{ \beta[a_1, \dotsc, a_n] = {\rm diag}(a_1,\dotsc,a_n, a_n^{-1},\dotsc,a_1^{-1}) \mid a_i>0 \}  & \mbox{ if } G= {\rm Sp}_{2n}(\R)\hbox{ or }{\rm SO}_{2n}(\R).
\end{cases}
\end{align*}
Then the Iwasawa decomposition $ G = UAK $ holds.

Let $ B $ be a Borel subgroup of $G$ with Langlands decomposition $B = MAU$,
where
\begin{align*}
M= \begin{cases}
     \{ \alpha [ \varepsilon_1,\dotsc,\varepsilon_n ] \mid \varepsilon_i \in \{\pm 1\} \} & \mbox{ if } G = {\rm GL}_n(\R), \\
 \{ \beta [ \varepsilon_1,\dotsc,\varepsilon_n ] \mid \varepsilon_i \in \{\pm 1\} \} & \mbox{ if } G= {\rm Sp}_{2n}(\R)\hbox{ or }{\rm SO}_{2n}(\R). \\
   \end{cases}
\end{align*}
For $ \nu = (\nu_1,\dotsc, \nu_n) \in \C^n $, define the quasi-character $e^{\nu} $ on $ A $  by 
\[
e^{\nu} (   \alpha[ a_1,\dotsc,a_n ] ) = \prod_{i=1}^n a_i^{\nu_i}, \qquad 
 e^{\nu} (   \beta[ a_1,\dotsc,a_n ] ) = \prod_{i=1}^n a_i^{\nu_i}
\]
and extend it to a character on $B$ by 
\[
   e^{\nu} (mau) =e^{\nu}(a),  \qquad  \forall (m,a,u) \in M\times A \times U.
\]
Let
\begin{align*}
 \rho = 
\begin{cases}
 (\tfrac{n-1}{2}, \tfrac{n-3}{2}, \dotsc, \tfrac{1-n}{2}) & \text{if $ G= {\rm GL}_n(\R)$}, \\
 (n-1, n-2, \dotsc, 0) & \text{if $ G= {\rm SO}_{2n}(\R) $}, \\
  (n, n-1, \dotsc, 1) & \text{if $ G=  {\rm Sp}_{2n}(\R)$}.
\end{cases}
\end{align*}
The (unnormalized) induced representation 
\[
 \pi_{\nu} = {\rm Ind}_B^G(1_M \otimes e^{\nu+\rho} \otimes 1_U) 
\] is called the spherical principal series representation of $G$.

\medskip 

To introduce the notion of a Whittaker function, we define a nondegenerate unitary character $ \psi $ on $U$ by
\begin{align*}
\psi(u)  = \begin{cases}
 \exp \{ 2\pi \sqrt{-1} (u_{1,2}+ \cdots + u_{n-1,n}) \} & \text{if $ G = {\rm GL}_n(\R)$}, \\
 \exp \{ 2\pi \sqrt{-1} (u_{1,2}+ \cdots + u_{n-1,n} + u_{n,n+1}) \} & \text{if $ G= {\rm Sp}_{2n}(\R)$}, \\
 \exp \{ 2\pi \sqrt{-1} (u_{1,2}+ \cdots + u_{n-1,n} + u_{n-1,n+1}) \} & \text{if $ G= {\rm SO}_{2n}(\R) $}
\end{cases}
\end{align*}
for $ u = (u_{i,j}) \in U $.

Let $W(\psi)$ be the space of smooth functions $w\colon G\to\mathbb{C}$
satisfying
\[
  w(ug)=\psi(u)w(g),  \qquad (u,g)\in U\times G.
\]
We equip $W(\psi)$ with the right regular action of $G$, defined by
\[
  (R(h)w)(g)=w(gh),  \qquad g,h\in G.
\]
Let $\C_{\psi}$ denote the one-dimensional representation of $U$
on which $u\in U$ acts by multiplication by $\psi(u)$. 
A $\psi$-Whittaker functional on $\pi_{\nu}$ is a nonzero element
$  \Lambda_{\nu}\in
  \operatorname{Hom}_{U} (\pi_{\nu}, {\C}_{\psi}), 
$
thus,
\[
  \Lambda_{\nu}\bigl(\pi_{\nu}(u)v\bigr)  =\psi(u)\Lambda_{\nu}(v),
  \qquad u\in U,\quad v\in\pi_{\nu}.
\]
For those $\nu$ for which $\pi_{\nu}$ is irreducible and generic, the
multiplicity-one theorem for Whittaker models gives
\[
  \dim \operatorname{Hom}_{U} (\pi_{\nu}, \C_{\psi})=1.
\]

Fix a nonzero $\psi$-Whittaker functional $\Lambda_{\nu}$. For
$v\in\pi_{\nu}$, define
\[
  W_v(g)=\Lambda_{\nu} (\pi_{\nu}(g)v),  \qquad g\in G.
\]
Then
\[
  W_v(ug)=\psi(u)W_v(g),
\]
so that $W_v\in W(\psi)$. Moreover,
\[
  R(h)W_v=W_{\pi_{\nu}(h)v},  \qquad h\in G,
\]
and hence the map
\[
  \pi_{\nu} \to W(\psi),
  \qquad v\mapsto W_v,
\]
is a $G$-intertwining map. We denote its image by
\[
  W(\pi_{\nu},\psi)
\]
and call it the Whittaker model of $\pi_{\nu}$. By the multiplicity-one
theorem, the functional $\Lambda_{\nu}$, and consequently all the functions
$W_v$, are uniquely determined up to a common nonzero scalar.

Since $\pi_{\nu}$ is a spherical representation, the space
$\pi_{\nu}^{K}$ of $K$-fixed vectors is one-dimensional. Take a nonzero vector $ v_0 \in \pi_{\nu}^K$.
The Whittaker function $W_{v_0}$ attached to $v_0$ is called the spherical
Whittaker function associated with $\pi_{\nu}$. 
We fix its overall scalar by the Mellin-Barnes formulas given in Propositions \ref{prop:exp_GLn}, \ref{prop:exp_SO2n}, and \ref{prop:exp_Sp2n}.

Since the spherical Whittaker function $W_{v_0}$ on $G$ is a right $K$-invariant function,
the Iwasawa decomposition implies that
$ W_{v_0} $ is determined by its restriction $W_{v_0}|_A $  to $A$.  We call this restriction  the {\it radial part} of $W_{v_0}$.

\medskip 

From now on we use $ \mu = (\mu_1,\dotsc,\mu_n) \in \C^n $ for the parameter of the spherical principal series representation of 
$ {\rm GL}_n(\R) $. 
In \cite{IS07}, \cite{I25a}, \cite{I25b}, Mellin-Barnes integral representations of the radial parts of the spherical Whittaker functions 
are given. Denote by $ W_{n,\mu}^A  $, $ W_{n,\nu}^C $ and $ W_{n,\nu}^D $ the spherical Whittaker functions 
on $ {\rm GL}_n(\R) $, ${\rm Sp}_{2n}(\R)$ and ${\rm SO}_{2n}(\R)$, respectively.  
We define $ \rho $-shifted Whittaker functions  
$ \widetilde{W}_{n,\mu}^A  $, $ \widetilde{W}_{n,\nu}^C $ and $ \widetilde{W}_{n,\nu}^D $ on $ A $ by 
\begin{align*}
 \widetilde{W}_{n,\mu}^A (a) &= e^{-\rho}(a) {W}_{n,\mu}^A(a), & a = \alpha[a_1,\dotsc,a_n] \in A, \\
 \widetilde{W}_{n,\nu}^C (a) &= e^{-\rho}(a) {W}_{n,\nu}^C(a), & a = \beta[a_1,\dotsc,a_n] \in A, \\
 \widetilde{W}_{n,\nu}^D (a) &= e^{-\rho}(a) {W}_{n,\nu}^D(a), & a = \beta[a_1,\dotsc,a_n] \in A,
\end{align*}
respectively.
Their Mellin-Barnes integral representations take the following forms:
\begin{align*} 
& \widetilde{W}_{n,\mu}^A (a) 
 = \frac{  a_n^{\mu_1+\cdots+\mu_n } }{ (2\pi \sqrt{-1})^{n-1}} \int_{s_1,\dotsc, s_{n-1}} V_{{n},\mu}^A(s_1,\dotsc,s_{n-1}) 
  \prod_{i=1}^{n-1} \Bigl( \frac{a_i}{a_{i+1}} \Bigr)^{-s_i}  \prod_{i=1}^{n-1} ds_i,
\end{align*}
for $ a =  \alpha[a_1,\dotsc, a_n] \in A $, 
and 
\begin{align*}
& \widetilde{W}_{n,\nu}^C (a)  
 = \frac{  1 }{ (2\pi \sqrt{-1})^n} \int_{s_1,\dotsc, s_n} V_{n,\nu}^C(s_1,\dotsc,s_n) 
  \prod_{i=1}^{n-1} \Bigl( \frac{a_i}{a_{i+1}} \Bigr)^{-s_i} \cdot  (a_n^2)^{-s_n} \prod_{i=1}^n ds_i,
\\
& \widetilde{W}_{n,\nu}^D (a) 
 = \frac{ 1 }{ (2\pi \sqrt{-1})^{n}} \int_{s_1,\dotsc, s_n} V_{n,\nu}^D(s_1,\dotsc,s_n) 
  \prod_{i=1}^{n-1} \Bigl( \frac{a_i}{a_{i+1}} \Bigr)^{-s_i} \cdot (a_{n-1}a_n)^{-s_n} \prod_{i=1}^n ds_i
\end{align*}
for $ a = \beta[a_1,\dotsc, a_n]   \in A $.
In the next two subsections, 
we give explicit formulas for the Mellin-Barnes kernels $V_{n,\mu}^A $, $V_{n,\nu}^C $, and $V_{n,\nu}^D$.

\subsection{Spherical Whittaker functions on ${\rm GL}_n(\R)$} \label{subsec:GLn_Wh}

In this subsection, we first recall explicit formulas for the spherical Whittaker functions on ${\rm GL}_n(\R)$, as 
given by the authors in \cite{IS07}. 

For $ \mu = (\mu_1,\dotsc, \mu_n) \in \C^n $, we set $ |\mu| = \sum_{i=1}^{n} \mu_i $.

\begin{prop} \label{prop:exp_GLn} $($\cite[Theorem 12]{IS07}, \cite[Proposition 2.4]{IS13}$)$
Let $ n \ge 3 $.
For $ 1 \le k \le n $, we have 
\begin{align*}
 V_{n,\mu}^A(s_1,\dotsc,s_{n-1})
& = \frac{2^{-1}}{ (4\pi \sqrt{-1})^{n-2}} \int_{t_1,\dotsc, t_{n-2}} V_{n-1,(\mu_1,\dotsc,\mu_{k-1}, \mu_{k+1}, \dotsc, \mu_n) }^A 
   (t_1,\dotsc,t_{n-2}) \\
& \times  \prod_{i=1}^{n-2} \GR(s_i-t_i) \cdot \GR(s_{n-1}+|\mu|-\mu_k) \prod_{i=1}^{n-1} \GR(s_i-t_{i-1}+\mu_k) 
   \, \prod_{i=1}^{n-2} dt_i,
\end{align*}
where we understand that $ t_0 = 0 $
and $ V_{2,\mu}^A(s_1) = 2^{-1} \GR(s_1+\mu_1) \GR(s_1+\mu_2). $
\end{prop}

The following formula was also useful in our previous paper \cite{IS13}, and 
it plays an important role in the present computation as well.

\begin{prop} \label{prop:Pieri_GLn} \cite[Proposition 2.6]{IS13} 
Let $ n \ge 2 $. 
For $ \sigma \in \C $, we have 
\begin{align*}
&  \prod_{i=1}^n \GR(\sigma+\mu_i) \cdot V_{n,\mu}^A(s_1,\dotsc,s_{n-1}) \\
& =  \GR(s_{n-1}+\sigma+|\mu|) \cdot  \frac{1}{ (4\pi \sqrt{-1})^{n-1}} \int_{t_1,\dotsc,t_{n-1}}
     V_{n,\mu}^A(t_1,\dotsc, t_{n-1}) \\
 & \times \prod_{i=1}^{n-1} \GR(s_i-t_i) \prod_{i=1}^{n-1} \GR(s_{i-1}-t_i + \sigma) \prod_{i=1}^{n-1} dt_i,
\end{align*}
where we understand that $ s_0=0 $.
\end{prop}

We now present a new formula that will be useful in the subsequent computations.

\begin{prop} \label{prop:new_GLn}
Let $ n \ge 2 $. For $\bs = (s_1,\dotsc,s_n) \in \C^n $
and $ s, \sigma, a_1,a_2  \in \C $, 
we define the integral $ T_{n,  \mu } (\bs ; s, \sigma,a_1,a_2) $ by 
\begin{align*}
& T_{n,  \mu } (\bs ; s, \sigma,a_1,a_2) \\
&= \frac{1}{ (4\pi \sqrt{-1})^{n-1}} \int_{t_1,\dotsc, t_{n-1}} 
   V_{ n,  (\mu_1,\dotsc, \mu_n) }^A (s-t_1, 2s-t_2, \dotsc, (n-1)s-t_{n-1}  ) 
\\
&\times 
  \prod_{i=1}^{n-1}  \frac{ \GR(t_i-s_i) \GR(t_i-s_{i+1}+s+a_2) \GR(t_{i-1}-s_i+s+a_1)  }
  { \GR(t_{i-1}+t_i-s_i -s_{i+1}+2s+a_1+a_2) } 
\\
& \times \prod_{i=1}^{n-1} \GR(t_{i-1}-s_{i+1}+2s+a_1+a_2) \\
& \times 
 \frac{ \GR(  t_{n-1} +s-\sigma +a_1 )\GR( -t_{n-1}+s+\sigma + a_2) }{  \GR(  (n+1)s - t_{n-1}+|\mu|+a_2 ) } 
  \prod_{i=1}^{n-1} dt_i, 
\end{align*}
where we understand that $ t_0= 0 $.
Then we have
\begin{align*}
& T_{n,\mu} (\bs; s, \sigma, a_1,a_2) \\
& = \frac{ 2\GR(s-s_{n-1}+\sigma+a_2) \GR(2s-s_n+\sigma+2a_2)}{ \GR(ns-\sigma +|\mu|) \prod_{i=1}^n \GR(2s+a_2+\mu_i) }
\\
& \times \frac{1}{ (4\pi \sqrt{-1})^{n-1}} \int_{r_1,\dotsc,r_{n-1}} 
V_{n+1, (\mu_1,\dotsc, \mu_n,a_1)}^A (s-r_1,2s-r_2, \dotsc, (n-1)s-r_{n-1}, ns-\sigma)
\\
& \times \prod_{i=1}^{n-1} \GR(r_i-s_i) 
   \prod_{i=1}^{n-1} \GR(r_i-s_{i+1}+s +a_2)
    \prod_{i=1}^{n-1} \GR( r_i-s_{i-1}+s+a_2) \\
& \times  \prod_{i=2}^{n-1} \GR(r_i-s_i+2s+2a_2) 
   \prod_{i=1}^{n-2}  \frac{1 }
  { \GR(r_i+r_{i+1}-s_i-s_{i+1}+2s+2a_2)}  \\
& \times   \frac{1}{\GR(r_{n-1}-s_{n-1}-s_n+2s+\sigma +2a_2)} 
 \,\prod_{i=1}^{n-1} dr_i,
\end{align*}
where we understand that $ s_0= 0 $.
\end{prop}
\bpf 
For $ 1 \le i \le n-1 $, Barnes' first lemma implies that 
\begin{align*}
& \frac{   \GR(t_i-s_i) \GR(t_i-s_{i+1}+s+a_2) \GR(t_{i-1}-s_i+s+a_1) \GR(t_{i-1}-s_{i+1}+2s+a_1+a_2) }
  { \GR(t_{i-1}+t_i-s_i -s_{i+1}+2s+a_1+a_2) }
\\
& = \frac{1}{4\pi \sqrt{-1}}\int_{r_i}  \GR(t_i-r_i)\GR(t_{i-1}-r_i +s+a_1)\GR(r_i-s_i )\GR(r_i-s_{i+1}+s +a_2) \,dr_i.
\end{align*}
We apply Proposition \ref{prop:Pieri_GLn} with $ \sigma = 2s+a_2 $ to get 
\begin{align*}
& \prod_{i=1}^{n} \GR(2s+a_2+\mu_i)  \cdot T_{n, \mu } (\bs; s, \sigma,a_1,a_2) \\
& = \frac{1}{ (4\pi \sqrt{-1})^{3n-3}} \int_{t_1,\dotsc, t_{n-1},  r_1,\dotsc, r_{n-1}, q_1,\dotsc, q_{n-1} } 
   V_{ n, \mu}^A (\bq) 
\\
& \times 
 \prod_{i=1}^{n-1} \GR(is-t_i-q_i) \prod_{i=1}^{n-1} \GR((i+1)s-t_{i-1}-q_i+a_2) 
\\
&\times 
  \prod_{i=1}^{n-1}   \GR(t_i-r_i) 
  \prod_{i=1}^{n-1}\GR(t_{i-1}-r_i+s +a_1)   
\\
& \times 
   \prod_{i=1}^{n-1}\GR(r_i-s_i ) \prod_{i=1}^{n-1}\GR(r_i-s_{i+1} +s+a_2)\\
& \times 
  \GR( t_{n-1} +s-\sigma +a_1 )\GR( -t_{n-1}+s+\sigma + a_2)  
  \prod_{i=1}^{n-1} dt_i  \prod_{i=1}^{n-1} dr_i \prod_{i=1}^{n-1} dq_i.
\end{align*}
Here we write $\bq = (q_1, \dotsc, q_{n-1}) $.
Applying Barnes' first lemma to the integrals in $t_1,t_2,\ldots,t_{n-1}$, we get\begin{align*}
& \prod_{i=1}^{n} \GR(2s+a_2+\mu_i)  \cdot T_{n, \mu} (\bs; s,\sigma,a_1,a_2) \\
& = \frac{1}{ (4\pi \sqrt{-1})^{2n-2}} \int_{r_1,\dotsc, r_{n-1}, q_1,\dotsc, q_{n-1} } 
   V_{ n, \mu}^A (\bq) 
  \GR(-q_1+2s+a_2) \GR(s-r_1 +a_1)
\\
& \times \prod_{i=1}^{n-2}
\frac{ \GR( is-r_i-q_i) \GR( (i+1)s-r_{i+1}-q_i+a_1) }
{\GR( (2i+3)s -r_i-r_{i+1}-q_i-q_{i+1} +a_1+a_2) }
\\
& \times \prod_{i=1}^{n-2} \GR((i+2)s-r_i-q_{i+1}+a_2) \prod_{i=1}^{n-2} \GR( (i+3)s- r_{i+1}-q_{i+1}+a_1+a_2)
\\
& \times 
\frac{\GR( (n-1)s-r_{n-1}-q_{n-1} ) \GR( ns-q_{n-1} -\sigma+a_1)
   }{ 
  \GR((n+1)s - r_{n-1} - q_{n-1} +a_1+a_2) }
\\
& \times \GR(s-r_{n-1}+\sigma +a_2 ) \GR(2s+a_1+ a_2)
\\
& \times 
   \prod_{i=1}^{n-1}\GR(r_i-s_i ) \prod_{i=1}^{n-1}\GR(s+r_i-s_{i+1} +a_2)
   \prod_{i=1}^{n-1} dr_i \prod_{i=1}^{n-1} dq_i
\\
& = \frac{ 1}{  (4\pi \sqrt{-1})^{n-1}} \int_{ q_1,\dotsc, q_{n-1} } 
     V_{ {n}, \mu}^A (\bq)  A_n(\bs, \bq; s, \sigma, a_1,a_2)  \prod_{i=1}^{n-1} dq_i,
\end{align*}
where $ A_n(\bs, \bq; s, \sigma, a_1,a_2) $ is the integral  \begin{align*}&A_n(\bs, \bq; s, \sigma, a_1,a_2)\\
& =\frac{1}{ (4\pi \sqrt{-1})^{n-1}} \int_{r_1,\dotsc, r_{n-1}} 
   \prod_{i=1}^{n-1}\GR( is-r_i-q_i) \prod_{i=1}^{n-1}\GR( is-r_i-q_{i-1}+a_1) \\
& \times \prod_{i=1}^{n-2} \frac{ 1 }{  \GR( (2i+3)s -r_i-r_{i+1}-q_i-q_{i+1} +a_1+a_2) }
\\
& \times \prod_{i=1}^{n-1} \GR((i+1)s-r_{i-1}-q_{i}+a_2) 
  \prod_{i=2}^{n-1} \GR( (i+2)s- r_{i}-q_{i}+a_1+a_2) 
\\
& \times 
  \frac{ \GR( ns-q_{n-1} -\sigma+a_1)
  \GR(s-r_{n-1}+\sigma +a_2 ) \GR(2s+a_1+ a_2)}{ \GR( (n+1)s-r_{n-1}-q_{n-1}+a_1+a_2)}  
\\
& \times 
   \prod_{i=1}^{n-1}\GR(r_i-s_i ) \prod_{i=1}^{n-1}\GR(s+r_i-s_{i+1} +a_2)
   \prod_{i=1}^{n-1} dr_i.
\end{align*}
Here we understand $ r_0 =q_0 = 0$.

To this result, we will apply:
\begin{lem} \label{lem:A=B}
We have
\begin{align} \label{eqn:A=B}
\begin{split}
&  A_n(\bs, \bq; s, \sigma, a_1,a_2) 
\\
& = \GR(s-s_{n-1}+\sigma+a_2) \GR(2s-s_n+\sigma+2a_2) \GR(ns-q_{n-1}-\sigma+a_1)
\\
& \times B_n(\bs, \bq; s, \sigma, a_1,a_2),   
\end{split}
\end{align}
where $ B_n(\bs, \bq; s, \sigma, a_1,a_2) $ is the integral  \begin{align*}&B_n(\bs, \bq; s, \sigma, a_1,a_2) \\
&= \frac{1}{ (4\pi \sqrt{-1})^{n-1} }\int_{r_1,\dotsc,r_{n-1}} 
 \prod_{i=1}^{n-1} \GR(is-r_i-q_i) \prod_{i=1}^{n-1} \GR(is-r_i-q_{i-1}+a_1) 
\\
& \times  \prod_{i=1}^{n-1} \GR(r_i-s_i) \prod_{i=1}^{n-1} \GR(s+r_i-s_{i+1}+a_2)  \prod_{i=1}^{n-1} \GR(s+r_i-s_{i-1}+a_2) \\
& \times  
\prod_{i=2}^{n-1} \GR(2s+ r_i-s_i+2a_2)  \prod_{i=1}^{n-2}  \frac{1}{ \GR(2s+r_i+r_{i+1}-s_i-s_{i+1}+2a_2)} 
\\
& \times \frac{ 1}{ \GR(2s+r_{n-1}-s_{n-1}-s_n+\sigma+2a_2) }
\prod_{i=1}^{n-1} dr_i.
\end{align*}
\end{lem}

\bpf 
We show (\ref{eqn:A=B}) by induction on $n$. 
For the case $n=2$, we see that
\begin{align*}
& A_2(\bs, \bq; s, \sigma, a_1,a_2) \\
& = \GR(2s-q_1+a_2) \GR( 2s-q_1 -\sigma+a_1) \GR(2s+a_1+ a_2) \\
& \times    I_1( s+\sigma+a_2,  s+a_1,  s-q_1;  -s_1,  s-s_2+a_2;   3s-q_1+a_1+a_2).
\end{align*} 
By Lemma \ref{lem:int1}(i), then,
\begin{align*}
& A_2(\bs, \bq; s, \sigma, a_1,a_2) \\
& = \GR( 2s-q_1 -\sigma+a_1) \GR(s-s_1+\sigma+a_2) \GR(2s-s_2+\sigma+2a_2) 
\\
& \times I_1(-s_1, s-s_2+a_2, s+a_2;   s+a_1, s-q_1;  2s-s_1-s_2+\sigma+2a_2),
\end{align*} 
as desired.

For our inductive argument, we assume that the identity (\ref{eqn:A=B}) holds for $n-1 $. 
We apply Lemma \ref{lem:int1}(i), with 
\begin{align*}
 (a,b,c) & = (s+\sigma+a_2,  (n-1)s-q_{n-2}+a_1,  (n-1)s-q_{n-1}), \\
 (d,e,f) & = (-s_{n-1},  s-s_n+a_2,  (2n-1)s-r_{n-2}-q_{n-2}-q_{n-1}+a_1+a_2), \end{align*}to the above integral in $r_{n-1}$, from which it follows that
\begin{align*}&A_n(\bs, \bq; s, \sigma, a_1,a_2)\\
& = \frac{1}{ (4\pi \sqrt{-1})^{n-1}} \int_{r_1,\dotsc, r_{n-1}} 
   \prod_{i=1}^{n-2}\GR( is-r_i-q_i) \prod_{i=1}^{n-2} \GR( is-r_i-q_{i-1}+a_1)
\\
& \times 
\prod_{i=1}^{n-3} \frac{ 1 }{  \GR( (2i+3)s -r_i-r_{i+1}-q_i-q_{i+1} +a_1+a_2) }
\\
& \times \prod_{i=1}^{n-2} \GR((i+1)s-r_{i-1}-q_{i}+a_2) \prod_{i=2}^{n-2} \GR( (i+2)s- r_{i}-q_{i}+a_1+a_2)
\\
& \times 
  \GR( ns-q_{n-1} -\sigma+a_1)
\GR(2s+a_1+ a_2)
\\
& \times 
   \prod_{i=1}^{n-2}\GR(r_i-s_i ) \prod_{i=1}^{n-2}\GR(s+r_i-s_{i+1} +a_2)
\\
& \times 
  \frac{ \GR(s-s_{n-1}+\sigma+a_2) \GR(2s-s_n+\sigma+2a_2) }{  \GR( ns -r_{n-2}-q_{n-2} +a_1+a_2) } 
\\
& \times \frac{ \GR(r_{n-1} -s_{n-1})\GR(s+r_{n-1} -s_n+a_2 ) \GR(s-r_{n-2}+r_{n-1} +a_2) }{ 
  \GR(2s+r_{n-1}-s_{n-1}-s_n+\sigma+2a_2)}
\\
& \times \GR(  (n-1)s-r_{n-1}  -q_{n-2}+a_1) \GR((n-1)s-r_{n-1} -q_{n-1}) 
   \prod_{i=1}^{n-1} dr_i.
\end{align*}
Here we replaced $ r_{n-1} \to -r_{n-1} $. Then we see that
\begin{align*}
& A_n(\bs, \bq; s, \sigma, a_1,a_2) \\
& = \GR(s-s_{n-1}+\sigma+a_2) \GR(2s-s_n+\sigma+2a_2) 
\\
& \times \frac{1}{ 4\pi \sqrt{-1} } \int_{r_{n-1}} 
   \GR(r_{n-1} -s_{n-1}) \GR(s+r_{n-1} -s_n+a_2 ) \\
& \times \frac{\GR((n-1)s-r_{n-1} -q_{n-1}) \GR( ns-q_{n-1} -\sigma+a_1)}{  \GR(2s+r_{n-1}-s_{n-1}-s_n+\sigma+2a_2)} 
\\
& \times A_{n-1}( (s_1,\dotsc, s_{n-1}), (q_1,\dotsc, q_{n-2}); s, r_{n-1}, a_1,a_2) \, dr_{n-1}.
\end{align*}
Applying the induction hypothesis, we obtain (\ref{eqn:A=B}), and our lemma is proved.  
\epf

\medskip 

Let us return to the proof of Proposition \ref{prop:new_GLn}.
Using Lemma \ref{lem:A=B}, we find that  
\begin{align*}
& \prod_{i=1}^{n} \GR(2s+a_2+\mu_i)  \cdot T_{n, \mu} (\bs; s,\sigma,a_1,a_2) 
\\
& = \GR(s-s_{n-1}+\sigma+a_2) \GR(2s-s_n+\sigma+2a_2)
\\
& \times 
  \frac{1}{ (4\pi \sqrt{-1})^{n-1}} \int_{r_1,\dotsc,r_{n-1}} 
    \prod_{i=1}^{n-1} \GR(r_i-s_i)
    \prod_{i=1}^{n-1} \GR(s+r_i-s_{i+1}+a_2) \\
& \times  \prod_{i=1}^{n-1} \GR(s+r_i-s_{i-1}+a_2)  \prod_{i=2}^{n-1} \GR(2s+ r_i-s_i+2a_2)  \\
& \times\prod_{i=1}^{n-2} \frac{1}{  \GR(2s+r_i+r_{i+1}-s_i-s_{i+1}+2a_2) }
   \cdot \frac{1}{ \GR(2s+r_{n-1}-s_{n-1}-s_n+\sigma+2a_2) }
\\
& \times 
  \frac{1}{(4\pi \sqrt{-1})^{n-1}} \int_{q_1,\dotsc,q_{n-1}}
    V_{ {n},\mu}^A (\bq)  \GR(ns-q_{n-1}-\sigma+a_1) 
\\
& \times 
    \prod_{i=1}^{n-1} \GR(is-r_i-q_i) \prod_{i=1}^{n-1} \GR(is-r_i-q_{i-1}+a_1)   \prod_{i=1}^{n-1} dq_i \prod_{i=1}^{n-1} dr_i.
\end{align*}
In view of Proposition \ref{prop:exp_GLn}, we can perform the integrations over $ q_i $ $ (1\le i \le n-1)$,
yielding the desired result.
\epf

\subsection{Spherical Whittaker functions on $ {\rm SO}_{2n}(\R) $ and $ {\rm Sp}_{2n}(\R) $}  
\label{subsec:SO2nSp2n_Wh}
In this subsection, we present the explicit formulas for
$ V_{n,\nu}^{C} $ and $ V_{n,\nu}^D $ established in the papers \cite{I25a} and \cite{I25b}, 
and further derive a new formula for $ V_{n,\nu}^D $ that will be useful in the computation of
archimedean zeta integrals.

\begin{prop} $($\cite[Theorem 4.1]{I25b}$)$ \label{prop:exp_SO2n}
Let $ n \ge 3 $. For $ \bs = (s_1,\dotsc,s_n) \in \C^n $ and $ \bt = (t_1,\dotsc,t_{n-1}) \in \C^{n-1} $, we have 
\begin{align*}
 V_{n,\nu}^D(\bs)
& = \frac{1 }{ (4\pi \sqrt{-1})^{n-1}} \int_{t_1,\dotsc, t_{n-1}} V_{n-1, (\nu_1,\dotsc,\nu_{n-1}) }^D (\bt) 
    L_{n, \nu_n}(\bs, \bt) \prod_{i=1}^{n-1} dt_i,
\end{align*}
where 
\begin{align*}
   L_{n,\nu_n} (\bs, \bt)
& = \prod_{i=1}^{n-3} \frac{\GR(s_i-t_i) \GR(s_{i+1}-t_{i-1}) }{ \GR(s_i+s_{i+1}-t_{i-1}-t_i)} 
  \prod_{i=1}^{n-2}  \GR(s_i-t_{i-1}+\nu_{n}) \prod_{i=1}^{n-2}  \GR(s_i-t_{i-1} -\nu_{n}) 
\\
& \times \frac{ \GR(s_{n-2}-t_{n-2}-t_{n-1}) \GR(s_{n-1}+s_{n}-t_{n-3})
    \GR(s_{n-1}-t_{n-2} + \frac{\nu_n}{2} )   }{ 
   \GR(s_{n-2}+s_{n-1}+s_{n} - t_{n-3}-t_{n-2} -t_{n-1})} 
\\
& \times 
 \GR(s_{n-1}-t_{n-1} - \tfrac{\nu_n}{2})  \GR(s_{n}-t_{n-2} - \tfrac{\nu_n}{2} )  \GR(s_{n}-t_{n-1} + \tfrac{\nu_n}{2}),
\end{align*}
and 
\begin{align*}
 V_{2, (\nu_1,\nu_2)}^D (s_1,s_2) & = 
   \GR(s_1+\tfrac{\nu_1-\nu_2}{2})\GR(s_1-\tfrac{\nu_1-\nu_2}{2}) 
   \GR(s_2+\tfrac{\nu_1+\nu_2}{2})\GR(s_2-\tfrac{\nu_1+\nu_2}{2}).
\end{align*}
Here we understand that $ t_0 = 0$.
\end{prop}

The Mellin-Barnes kernel $ V_{n,\nu}^C $ is  expressible as an integral involving $V_{n,\nu}^D $, via

\begin{prop} $($\cite[Corollary 2.4]{I25a}$)$ \label{prop:exp_Sp2n}
Let $ n \ge 2 $. For $ \bs = (s_1,\dotsc,s_n) \in \C^n $ and $ \bt = (t_1,\dotsc,t_{n}) \in \C^{n} $, we have 
\begin{align*}
 V_{n,\nu}^C(\bs) 
& = \frac{1}{ (4\pi \sqrt{-1})^n } \int_{t_1,\dotsc,t_n} V_{n,\nu}^D(\bt) \prod_{i=1}^{n-2} \GR(s_i-t_i) \\
& \times   \GR(s_{n-1}-t_{n-1}-t_n) 
    \GR(s_n-t_n) \prod_{i=1}^n \GR(s_i-t_{i-1})   \prod_{i=1}^n dt_i,
\end{align*}
where we understand $ t_0 = 0 $.
\end{prop}

The following proposition, which is an ${\rm SO}_{2n}$-analog of Proposition \ref{prop:Pieri_GLn}, plays a crucial role in this paper.

\begin{prop} \label{prop:Pieri_SO2n}
Let $ n\ge 2 $. For $ \bs = (s_1,\dotsc,s_n) \in \C^n $, $ \bt = (t_1,\dotsc,t_{n}) \in \C^{n} $ and $\sigma \in \C$, we have
\begin{align*}
& \prod_{i=1}^n \GR(\sigma + \nu_i) \prod_{i=1}^n \GR(\sigma-\nu_i)  \cdot V_{n, \nu}^D (\bs)
\\
& = \GR(2\sigma) \cdot \frac{1}{(4\pi \sqrt{-1})^{n}} \int_{t_1,\dotsc,t_n} V_{n,\nu}^D (\bt) 
    M_n(\bs, \bt; \sigma)  \prod_{i=1}^n dt_i,
\end{align*}
where
\begin{align*}
 M_n(\bs, \bt; \sigma)
& =
 \prod_{i=1}^{n} \GR(s_i-t_i)  \prod_{i=2}^{n-2} \GR(s_i-t_i+2\sigma)\\
& \times  \prod_{i=1}^{n-2} \GR(s_{i-1} -t_i + \sigma)   \cdot \GR(s_{n-2} -t_{n-1}-t_n + \sigma) \GR(s_{n-1}-t_{n}+\sigma) 
\\
& \times  \prod_{i=1}^{n-2} \GR(s_{i}-t_{i-1} +\sigma) \cdot  \GR(s_{n-1}+s_n-t_{n-2}+\sigma) \GR(s_n-t_{n-1}+\sigma)   
\\
& \times \prod_{i=1}^{n-3} \frac{1}{  \GR(s_i+s_{i+1}-t_i-t_{i+1} +2\sigma) }
\\
& \times \frac{ 1}{ \GR(s_{n-2}+s_{n-1}+s_{n} -t_{n-2}-t_{n-1}-t_n+2\sigma) }
 \end{align*}
and we understand that $ s_0 = t_0= 0 $.
\end{prop}
\bpf
We prove the assertion by induction on $n$. The case $n=2$ is a consequence of Barnes' second lemma, applied twice in succession:

\begin{align*}
& \frac{1}{ (4\pi \sqrt{-1})^2} \int_{t_1,t_2} V_{2,(\nu_1,\nu_2)}^D(t_1,t_2) M_2((s_1,s_2), (t_1,t_2);\sigma) \,dt_1dt_2
\\
& = \frac{ 1 }{ (4\pi \sqrt{-1})^2} \int_{t_1,t_2}
      \GR(t_1+\tfrac{\nu_1-\nu_2}{2})\GR(t_1-\tfrac{\nu_1-\nu_2}{2}) 
   \GR(t_2+\tfrac{\nu_1+\nu_2}{2})\GR(t_2-\tfrac{\nu_1+\nu_2}{2})
\\
& \times \frac{ \GR(s_1-t_1) \GR(s_2-t_2) \GR(-t_1-t_2+\sigma) \GR(s_1-t_2+\sigma) }{ \GR(s_1+s_2-t_1-t_2+2\sigma)} \\
& \times \GR(s_1+s_2+\sigma) \GR(s_2-t_1+\sigma) \,dt_1dt_2 \\
& = \GR(s_1+\tfrac{\nu_1-\nu_2}{2}) \GR(s_1-\tfrac{\nu_1-\nu_2}{2}) \GR( s_2+\tfrac{\nu_1-\nu_2}{2}+\sigma)  
    \GR( s_2-\tfrac{\nu_1-\nu_2}{2}+\sigma) \\
& \times 
\frac{1}{4\pi \sqrt{-1}}\int_{t_2} \frac{ \GR(-t_2+\tfrac{\nu_1-\nu_2}{2}+\sigma) \GR(-t_2-\tfrac{\nu_1-\nu_2}{2}+\sigma) 
 \GR(s_2-t_2)  }{  \GR(s_2-t_2+2\sigma) } 
\\
& \times \GR(t_2+\tfrac{\nu_1+\nu_2}{2})\GR(t_2-\tfrac{\nu_1+\nu_2}{2}) \,dt_2
\\
& = \GR(s_1+\tfrac{\nu_1-\nu_2}{2}) \GR(s_1-\tfrac{\nu_1-\nu_2}{2})  \GR(s_2+ \tfrac{\nu_1+\nu_2}{2}) \GR(s_2-\tfrac{\nu_1+\nu_2}{2}) \\
& \times 
   \frac{ \GR(\sigma+\nu_1) \GR(\sigma+\nu_2)\GR(\sigma-\nu_1) \GR(\sigma-\nu_2)}{\GR(2\sigma)}.
\end{align*}
This proves the case $ n=2 $. 

Now assume that $n\ge 3 $, and that the assertion holds for $n-1$.  We will need the following lemma, which  plays a key role in the proof of Proposition \ref{prop:Pieri_SO2n}.
\begin{lem} \label{lem:keylemma_Pieri_SO2n}
Let  $ n \ge 3 $.
For $ \bs= (s_1,\dotsc, s_n) \in \C^n$, $ \bt = (t_1,\dotsc, t_n) \in \C^n$,
$ \bu = (u_1,\dotsc,u_{n-1})\in \C^{n-1} $ and $ \bv = (v_1,\dotsc,v_{n-1})  \in \C^{n-1} $, we have
\begin{align} \label{eqn:keylemma_Pieri_SO2n}
\begin{split}
& \frac{1}{ (4\pi \sqrt{-1})^{n}} 
  \int_{s_1,\dotsc, s_n} L_{n,\nu_n} (\boldsymbol{s}, \boldsymbol{u}) M_n(\boldsymbol{t}, \boldsymbol{s}; \sigma) \prod_{i=1}^n ds_i
\\
& = \GR(\sigma+\nu_n) \GR(\sigma-\nu_n) \cdot
   \frac{1}{ (4\pi \sqrt{-1})^{n-1} } \int_{v_1,\dotsc,v_{n-1}} L_{n,\nu_n}(\boldsymbol{t}, \boldsymbol{v} ) 
  M_{n-1} (\boldsymbol{v},  \boldsymbol{u}; \sigma) \prod_{i=1}^{n-1} dv_i,
\end{split}
\end{align}
where we understand that $ s_0 = t_0 = u_0 = v_0 = 0$.
\end{lem}

Let us, for the moment, assume Lemma \ref{lem:keylemma_Pieri_SO2n}. We may then apply 
Proposition \ref{prop:exp_SO2n} and the induction hypothesis, to obtain 
\begin{align*}
& \frac{1}{(4\pi \sqrt{-1})^{n}} \int_{t_1,\dotsc,t_n} V_{n,\nu}^D  (\bt) 
    M_n(\bs, \bt; \sigma)  \prod_{i=1}^n dt_i
\\
& = \frac{1 }{(4\pi \sqrt{-1})^{2n-1}}  \int_{t_1,\dotsc, t_n, u_1,\dotsc, u_{n-1}} 
    V_{n-1,(\nu_1,\dotsc, \nu_{n-1}) }^D (\bu) L_{n,\nu_n}(\bt, \bu) M_n(\bs, \bt; \sigma)  \prod_{i=1}^n dt_i \prod_{i=1}^{n-1} du_i
\\
& = \GR(\sigma+\nu_n) \GR(\sigma-\nu_n) \cdot \frac{1 }{(4\pi \sqrt{-1})^{n-1}}  \int_{v_1,\dotsc, v_{n-1}} 
   L_{n,\nu_n}(\bs, \bv) \\
& \times \biggl\{ \frac{1}{ (4\pi \sqrt{-1})^{n-1}} \int_{u_1,\dotsc,u_{n-1}} V_{n-1,(\nu_1,\dotsc, \nu_{n-1}) }^D (\bu)  
         M_{n-1}(\bv, \bu; \sigma)   \prod_{i=1}^{n-1} du_i \biggr\} \prod_{i=1}^{n-1} dv_i
\\
& = \frac{ \prod_{i=1}^n \GR(\sigma + \nu_i) \prod_{i=1}^n \GR(\sigma-\nu_i) }{ \GR(2\sigma) } \\
& \times 
     \frac{ 1 }{(4\pi \sqrt{-1})^{n-1}}  \int_{v_1,\dotsc, v_{n-1}}  V_{n-1,(\nu_1,\dotsc, \nu_{n-1}) }^D (\bv)
   L_{n,\nu_n}(\bs, \bv)
   \prod_{i=1}^{n-1} dv_i
\\
& =       \frac{ \prod_{i=1}^n \GR(\sigma + \nu_i) \prod_{i=1}^n \GR(\sigma-\nu_i) }{ \GR(2\sigma) }  \cdot 
    V_{n,\nu }^D (\bs),
\end{align*}
as desired.\epf

\medskip 

\noindent 
{\it Proof of Lemma \ref{lem:keylemma_Pieri_SO2n}.}
We will demonstrate the identity (\ref{eqn:keylemma_Pieri_SO2n}) by applying Lemma \ref{lem:int3} to the integral in each of the following variables, in the following order:
$$
 s_1 \to s_2 \to \cdots \to s_{n-3} \to s_{n-2} \to s_{n-1} \to s_{n} \to s_{n-2} \to s_{n-3} \to \cdots \to s_3 \to s_2.
$$
We proceed in seven steps.  In Step $i$, for $1\le i\le 6$, we will obtain an expression of the form  \begin{align} \label{eqn:kernel_K}
\begin{split}
&  \frac{1}{ (4\pi \sqrt{-1})^{n}} 
  \int_{s_1,\dotsc, s_n} L_{n,\nu_n}(\boldsymbol{s}, \boldsymbol{u}) M_n(\boldsymbol{t}, \boldsymbol{s}; \sigma) \prod_{i=1}^n ds_i 
\\
& = 
 \GR(\sigma+\nu_n) \cdot \frac{1}{ (4\pi \sqrt{-1})^{n}} \int_{s_1,\dotsc, s_n} K_i(\bs, \bu, \bt; \sigma) \prod_{i=1}^{n} ds_i,
\end{split}
\end{align}for an appropriate function $K_i(\bs, \bu, \bt; \sigma)$.  Then, in Step 7, we use some changes of variable and an application of Barnes' second lemma to obtain the stated result.

\noindent
{\it Step 1}.
If $n=3$, we can skip this step and proceed directly to Step 2.

If $n>3$, we apply Lemma \ref{lem:int3}  to the integrals in $s_1,s_2,\ldots, s_{n-3}$  successively, with 
\begin{align*}
 (a,b,c,d) & = ( -u_{i-1}+\nu_n,   -u_i, -s_{i-1}-\nu_n,   s_{i+1}-u_{i-1}-u_i), \\
 (\alpha,\beta, \gamma, \delta) & = (t_i,  t_{i+1}+\sigma,   s_{i-1}+\sigma,  -s_{i+1}+t_i+t_{i+1} +2\sigma)
\end{align*}
for the integral in $s_i$ ($1\le i\le n-3$).

Let us consider the integral in $s_1$. Since 
\begin{align*}
&  I_3( \nu_n,  -u_1, -\nu_n,   s_2-u_1;   t_1, t_2+\sigma, \sigma,  -s_2+t_1+t_2+2\sigma)  \\
& = \frac{ \GR(\sigma+\nu_n) \GR(-u_1+\sigma) \GR(t_1-\nu_n) \GR(t_2+\sigma-\nu_n)}{
   \GR(s_2-u_1-\nu_n) \GR(s_2) \GR(t_2-s_2+2\sigma) \GR(t_1-s_2+\sigma)}
\\
& \times I_3( \nu_n,  -u_1,  -s_2+\sigma, -u_1+\sigma+\nu_n;   t_1, t_2+\sigma,  s_2-\nu_n,  t_1+t_2+\sigma-\nu_n),
\end{align*}
application of Lemma 1.4 to this integral  yields\begin{align*}
&  \frac{1}{ (4\pi \sqrt{-1})^{n}} 
  \int_{s_1,\dotsc, s_n} L_{n,\nu_n}(\boldsymbol{s}, \boldsymbol{u}) M_n(\boldsymbol{t}, \boldsymbol{s}; \sigma)  \prod_{i=1}^n ds_i
\\
& =  \frac{1}{ (4\pi \sqrt{-1})^{n}} 
  \int_{s_1,\dotsc, s_n} L_{n,\nu_n} (\boldsymbol{s}, \boldsymbol{u}) M_n(\boldsymbol{t}, \boldsymbol{s}; \sigma)
\\
& \times \frac{ \GR(\sigma+\nu_n) \GR(-u_1+\sigma) \GR(t_1-\nu_n) \GR(t_2+\sigma-\nu_n)}{
   \GR(s_2-u_1-\nu_n) \GR(s_2) \GR(t_2-s_2+2\sigma) \GR(t_1-s_2+\sigma)}
\\
& \times \frac{  \GR(s_1-s_2+\sigma) \GR(s_1+s_2-u_1) \GR(-s_1+s_2-\nu_n) \GR(t_1+t_2-s_1-s_2+2\sigma)}
{ \GR( s_1-\nu_n) \GR(s_1-u_1+\sigma+\nu_n) \GR(-s_1+\sigma) \GR(t_1+t_2-s_1+\sigma-\nu_n)}
 \prod_{i=1}^n ds_i.
\end{align*}
We repeat this process for the remaining $s_i$'s ($2\le i\le n-3$); after treatment of the integral in $s_{n-3}$, we get\begin{align*}&  \frac{1}{ (4\pi \sqrt{-1})^{n}} 
  \int_{s_1,\dotsc, s_n} L_{n,\nu_n}(\boldsymbol{s}, \boldsymbol{u}) M_n(\boldsymbol{t}, \boldsymbol{s}; \sigma)  \prod_{i=1}^n ds_i
\\
& =  
 \frac{1}{ (4\pi \sqrt{-1})^{n}} 
  \int_{s_1,\dotsc, s_n} L_{n,\nu_n}(\boldsymbol{s}, \boldsymbol{u}) M_n(\boldsymbol{t}, \boldsymbol{s}; \sigma)
\\
& \times 
 \prod_{i=1}^{n-3} \frac{ \GR(s_{i-1}-u_{i-1}+\sigma+\nu_n)  \GR(s_{i-1}-u_i+\sigma)}{ \GR(s_{i+1}-u_i-\nu_n) \GR(s_{i+1}-u_{i-1})}
\\
& \times \prod_{i=1}^{n-3} 
   \frac{  \GR(t_i-s_{i-1}-\nu_n) \GR(t_{i+1}-s_{i-1}+\sigma-\nu_n)}
{  \GR(t_{i+1}-s_{i+1} +2\sigma) \GR(t_i-s_{i+1}+\sigma) }
\\
& \times \prod_{i=1}^{n-3} 
   \frac{ \GR(s_i-s_{i+1}+\sigma) \GR(s_i+s_{i+1}-u_{i-1}-u_i) }{\GR(-s_{i-1}+s_i-\nu_n) \GR(s_{i-1}+s_i-u_{i-1}-u_i+\sigma+\nu_n)}
\\
& \times \prod_{i=1}^{n-3}
 \frac{ \GR(-s_i+s_{i+1}-\nu_n) \GR(t_i+t_{i+1}-s_i-s_{i+1}+2\sigma) }{\GR(s_{i-1}-s_i+\sigma)\GR(t_i+t_{i+1}-s_{i-1}-s_{i}+\sigma-\nu_n) }
 \prod_{i=1}^n ds_i.
\end{align*}
Then the integral kernel $ K_1 (\bs, \bu, \bt; \sigma) $ in (\ref{eqn:kernel_K}) is  
\begin{align*}& K_1 (\bs, \bu, \bt; \sigma) \\
&=   \prod_{i=1}^{n-3} \GR(s_i-u_i)  
   \prod_{i=1}^{n-3}  \GR(s_{i-1}-u_i+\sigma) 
   \prod_{i=1}^{n-2} \GR(s_i-u_{i-1}+\nu_{n})  
      \prod_{i=1}^{n-4} \GR(s_{i}-u_{i}+\sigma+\nu_n) 
\\
& \times \prod_{i=1}^{n-3} \frac{1}{ \GR(s_{i-1}+s_{i}-u_{i-1}-u_{i}+\sigma+\nu_n)} \\
&\times   
 \GR(s_{n-3}-s_{n-2}+\sigma) \GR(s_{n-2}-s_{n-3}-\nu_n)    \GR(s_{n-2}-u_{n-2}-u_{n-1})  
\\
& \times \frac{ \GR(s_{n-1}+s_{n}-u_{n-3})
    \GR(s_{n-1}-u_{n-2} + \frac{\nu_n}{2} )  \GR(s_{n-1}-u_{n-1} - \tfrac{\nu_n}{2})  }{ 
   \GR(s_{n-2}+s_{n-1}+s_{n} - u_{n-3}-u_{n-2} -u_{n-1})} 
\\
& \times 
  \GR(s_{n}-u_{n-2} - \tfrac{\nu_n}{2} )  \GR(s_{n}-u_{n-1} + \tfrac{\nu_n}{2})  
\\
& \times \prod_{i=1}^{n} \GR(t_i-s_i) 
   \prod_{i=1}^{n-3} \frac{\GR( t_i-s_{i-1}-\nu_n) \GR(t_{i+1}-s_{i-1}+\sigma-\nu_n) }{ \GR(t_{i}+t_{i+1}-s_{i-1}-s_{i}+\sigma-\nu_n)} \\
& \times 
   \prod_{i=1}^{n-2} \GR(t_{i}-s_{i-1} +\sigma) \cdot \GR(t_{n-1}+t_n-s_{n-2}+\sigma)  \GR(t_n-s_{n-1}+\sigma) \\
& \times \frac{ \GR(t_{n-2} -s_{n-1}-s_n + \sigma)   
    \GR(t_{n-1}-s_{n}+\sigma)  }{ 
 \GR(t_{n-2}+t_{n-1}+t_{n} -s_{n-2}-s_{n-1}-s_n+2\sigma) }.
\end{align*}

\noindent
{\it Step 2}. We apply Lemma \ref{lem:int3} to the integral in $s_{n-2}$, with  
\begin{align*}
 (a,b,c,d) & = ( -u_{n-3}+\nu_n,   -u_{n-2}-u_{n-1}, -s_{n-3}-\nu_n,   s_{n-1}+s_n-u_{n-3}-u_{n-2}-u_{n-1}), \\
 (\alpha,\beta, \gamma, \delta) & = (t_{n-1}+t_{n}+\sigma,  t_{n-2},    s_{n-3}+\sigma,  t_{n-2}+t_{n-1}+t_n-s_{n-1}-s_n+2\sigma).
\end{align*}
Then the integral kernel $ K_2(\bs, \bu, \bt; \sigma) $ is 
\begin{align*}
& K_2(\bs, \bu, \bt; \sigma) =  \prod_{i=1}^{n-3} \GR(s_i-u_i)  \cdot \GR(s_{n-2}-u_{n-2}-u_{n-1}) \\
& \times   \prod_{i=1}^{n-3}  \GR(s_{i-1}-u_i+\sigma) \cdot \GR(s_{n-3} -u_{n-2}-u_{n-1}+\sigma)\\
& \times 
   \prod_{i=1}^{n-2} \GR(s_i-u_{i-1}+\nu_{n})  
      \prod_{i=1}^{n-3} \frac{\GR(s_{i}-u_{i}+\sigma+\nu_n) }{ \GR(s_{i-1}+s_{i}-u_{i-1}-u_{i}+\sigma+\nu_n)}
\\
&\times  \GR(-s_{n-2}+s_{n-1}+s_n-\nu_n)    \GR(s_{n-2}-s_{n-1}-s_n+\sigma) 
\\
& \times  \frac{  \GR(s_{n-1}-u_{n-2} + \frac{\nu_n}{2} )  \GR(s_{n-1}-u_{n-1} - \tfrac{\nu_n}{2})
    \GR(s_{n}-u_{n-2} - \tfrac{\nu_n}{2} )  \GR(s_{n}-u_{n-1} + \tfrac{\nu_n}{2})    }{ 
  \GR(s_{n-2}+s_{n-3}-u_{n-3}-u_{n-2}-u_{n-1}+\sigma+\nu_n) \GR(s_{n-1}+s_n-u_{n-2}-u_{n-1}-\nu_n)  }
\\
& \times \prod_{i=1}^{n} \GR(t_i-s_i) 
  \prod_{i=1}^{n-2}\GR( t_i-s_{i-1}-\nu_n) 
  \prod_{i=1}^{n-3} \frac{  \GR(t_{i+1}-s_{i-1}+\sigma-\nu_n)}{ \GR(t_{i}+t_{i+1}-s_{i-1}-s_{i}+\sigma-\nu_n)} \\
& \times 
   \prod_{i=1}^{n-2} \GR(t_{i}-s_{i-1} +\sigma) \cdot \GR(t_{n-1}+t_n-s_{n-2}+\sigma)  \GR(t_n-s_{n-1}+\sigma) \\
&  \times 
  \frac{ \GR(t_{n-1}-s_{n}+\sigma) \GR(t_{n-1}+t_{n}-s_{n-3}+\sigma-\nu_n) }{  
  \GR(t_{n-1}+t_n-s_{n-1}-s_n+2\sigma)\GR(t_{n-2}+t_{n-1}+t_n-s_{n-3}-s_{n-2} +\sigma-\nu_n) }.
\end{align*}

\medskip 

\noindent
{\it Step 3}. We apply Lemma \ref{lem:int3} to the integral in $s_{n-1}$, with  
\begin{align*}
 (a,b,c,d) & = ( -u_{n-1}-\tfrac{\nu_n}{2},   -s_{n-2}+s_n -\nu_n,  -u_{n-2}+\tfrac{\nu_n}{2},  s_n-u_{n-2}-u_{n-1} -\nu_n), \\
 (\alpha,\beta, \gamma, \delta) & = (t_{n}+\sigma,  s_{n-2}-s_n+\sigma,  t_{n-1},   t_{n-1}+t_n-s_n+2\sigma).
\end{align*}
Then the integral kernel $ K_3(\bs, \bu, \bt; \sigma) $ is  \begin{align*}
&K_3(\bs, \bu, \bt; \sigma)= \prod_{i=1}^{n-3} \GR(s_i-u_i)
   \prod_{i=1}^{n-3}  \GR(s_{i-1}-u_i+\sigma) 
\\
& \times    
   \prod_{i=1}^{n-2} \GR(s_i-u_{i-1}+\nu_{n})  
    \prod_{i=1}^{n-3} \frac{\GR(s_{i}-u_{i}+\sigma+\nu_n)}{      \GR(s_{i-1}+s_{i}-u_{i-1}-u_{i}+\sigma+\nu_n)} 
\\
& \times \frac{ \GR( s_{n-3}-u_{n-2}-u_{n-1}+\sigma) \GR(s_{n-2}- s_{n-1}-u_{n-1}) }{ 
    \GR(s_{n-3}+s_{n-2}  -u_{n-3}  -u_{n-2}-u_{n-1}+\sigma+\nu_n )  }
\\
& \times \GR(s_{n-2}-s_n-u_{n-2}+\sigma+\tfrac{\nu_n}{2} )  \GR(s_{n-1}-s_n+\sigma-\tfrac{\nu_n}{2}) 
\\
& \times   \GR(s_{n-1}  -u_{n-2} ) \GR(-s_{n-1}+s_n -\tfrac{\nu_n}{2})\GR(s_{n}-u_{n-1} + \tfrac{\nu_n}{2})
\\
& \times   \prod_{i=1}^{n-2} \GR(t_i-s_i)  \cdot \GR(t_{n-1}-s_{n-1}+\tfrac{\nu_n}{2})  \GR(t_n-s_n) \prod_{i=1}^{n-2} \GR( t_i-s_{i-1}-\nu_n) \\
& \times 
 \prod_{i=1}^{n-3} \frac{ \GR(t_{i+1}-s_{i-1}+\sigma-\nu_n) }{ \GR(t_{i}+t_{i+1}-s_{i-1}-s_{i}+\sigma-\nu_n)}
   \prod_{i=1}^{n-2} \GR(t_{i}-s_{i-1} +\sigma) 
\\
&  \times \frac{ \GR( t_{n-1}-s_{n-2}+s_n -\nu_n) \GR(t_{n-1}+t_{n}-s_{n-3}+\sigma-\nu_n)  }{ 
  \GR(t_{n-2}+t_{n-1}+t_{n}-s_{n-3}-s_{n-2}+\sigma -\nu_n ) } 
\\
& \times \frac{ \GR(t_{n}-s_{n-2}+s_{n-1}-\tfrac{\nu_n}{2}+\sigma)
  \GR(  t_{n-1}-u_{n-1}-\tfrac{\nu_n}{2}) \GR(t_{n}-u_{n-2}+\sigma+\tfrac{\nu_n}{2}) }
  { \GR(t_{n-1}-s_{n-1} +s_n-u_{n-1}-\nu_n ) \GR(t_{n}+s_{n-1}-s_n-u_{n-2}+2\sigma)}.
\end{align*}
Here we  substituted $ s_{n-1} \to -s_{n-1}+s_{n-2} + \tfrac{\nu_n}{2} $.

\medskip

\noindent
{\it Step 4}.
We apply Lemma \ref{lem:int3} to the integral in $s_n$, with 
\begin{align*}
 (a,b,c,d) & = (t_{n-1}-s_{n-2}-\nu_n,  -s_{n-1}-\tfrac{\nu_n}{2},  -u_{n-1}+ \tfrac{\nu_n}{2},   t_{n-1}-s_{n-1}-u_{n-1}-\nu_n), \\
 (\alpha,\beta, \gamma, \delta) & = (s_{n-2}-u_{n-2}+\sigma+\tfrac{\nu_n}{2},  s_{n-1}+\sigma-\tfrac{\nu_n}{2},  t_{n},  
    t_{n}+s_{n-1}-u_{n-2}+2\sigma).
\end{align*}
Then the integral kernel $ K_4(\bs, \bu, \bt; \sigma) $ is  \begin{align*}
& K_4(\bs, \bu, \bt; \sigma) = \prod_{i=1}^{n-3} \GR(s_i-u_i)    
   \prod_{i=1}^{n-3}  \GR(s_{i-1}-u_i+\sigma) \\
& \times 
     \prod_{i=1}^{n-2} \GR(s_i-u_{i-1}+\nu_{n})
    \prod_{i=1}^{n-3} \frac{ \GR(s_{i}-u_{i}+\sigma+\nu_n) }{ \GR(s_{i-1}+s_{i}-u_{i-1}-u_{i}+\sigma+\nu_n) }\\
& \times 
   \frac{ \GR( s_{n-3}-u_{n-2}-u_{n-1}+\sigma)   \GR(s_{n-2}-u_{n-2}-u_{n-1}+\sigma+ \nu_n) }{
    \GR(s_{n-3}+s_{n-2}  -u_{n-3}  -u_{n-2}-u_{n-1}+\sigma+\nu_n ) \GR(s_{n-1}+s_n-u_{n-2}-u_{n-1}+2\sigma) }
\\
& \times    
 \GR(s_{n-2}-s_{n-1}-s_n) \GR(-s_{n-2}+ s_{n-1}+s_n+\sigma-\nu_n) \\
& \times 
 \GR(s_{n-1}  -u_{n-2} ) 
 \GR(s_{n-1}-u_{n-1}+\sigma) 
  \GR(s_n-u_{n-2}+\sigma) 
  \GR(s_n -u_{n-1} ) 
 \\
& \times \prod_{i=1}^{n-2} \GR(t_i-s_i) \cdot   \GR(t_{n-1}-s_{n-1}+\tfrac{\nu_n}{2})  \GR(t_{n}-s_{n}+\tfrac{\nu_n}{2}) 
   \prod_{i=1}^{n-2}  \GR( t_i-s_{i-1}-\nu_n)
\\
& \times   \prod_{i=1}^{n-3} \frac{ \GR(t_{i+1}-s_{i-1}+\sigma-\nu_n) }{ \GR(t_{i}+t_{i+1}-s_{i-1}-s_{i}+\sigma-\nu_n)}
   \prod_{i=1}^{n-2} \GR(t_{i}-s_{i-1} +\sigma)
\\
& \times   \frac{  \GR(t_n-s_{n-1}-\tfrac{\nu_n}{2}) \GR(t_{n-1}+t_{n}-s_{n-3}+\sigma-\nu_n) \GR(t_{n-1}-s_n-\tfrac{\nu_n}{2}) }{ 
    \GR(t_{n-1}+t_n -s_{n-1}-s_n-\nu_n) } 
 \\
& \times 
  \frac{  \GR(t_{n-1}+t_n-s_{n-2}-\nu_n)}{  
   \GR(t_{n-2}+t_{n-1}+t_{n}-s_{n-3}-s_{n-2}+\sigma -\nu_n )}.
\end{align*}
Here we substituted $ s_n \to -s_n+s_{n-2}+\tfrac{\nu_n}{2} $.

Now if $n=3$, skip Steps 5 and 6, and proceed directly to Step 7.  Otherwise, the steps are as follows.

\noindent
{\it Step 5}. We apply Lemma \ref{lem:int3}  to the integral in $s_{n-2}$,  with 
\begin{align*}
 (a,b,c,d) & = (-u_{n-3}+\nu_n,  -u_{n-2}-u_{n-1} +\sigma+\nu_n,  -s_{n-1}-s_n,  \\
& \qquad  s_{n-3}-u_{n-3}-u_{n-2}-u_{n-1} +\sigma+\nu_n), \\
 (\alpha,\beta, \gamma, \delta) & = (t_{n-1}+t_n-\nu_n,  t_{n-2},    s_{n-1}+s_n+\sigma-\nu_n,  t_{n-2}+t_{n-1}+t_n-s_{n-3}+\sigma-\nu_n). 
\end{align*}
Then the integral kernel $ K_5(\bs, \bu, \bt; \sigma) $ is  \begin{align*}
& K_5(\bs, \bu, \bt; \sigma)=\prod_{i=1}^{n-4} \GR(s_i-u_i)  \prod_{i=1}^{n-3}  \GR(s_{i-1}-u_i+\sigma)      
   \prod_{i=1}^{n-2} \GR(s_i-u_{i-1}+\nu_{n})  
\\
& \times \prod_{i=1}^{n-3} \frac{ \GR(s_{i}-u_{i}+\sigma+\nu_n) }{\GR(s_{i-1}+s_{i}-u_{i-1}-u_{i}+\sigma+\nu_n)}
   \cdot \GR(s_{n-2}-u_{n-2}-u_{n-1} +\sigma+\nu_n)
\\
&  \times \frac{ \GR(s_{n-2}-s_{n-3}+\sigma ) \GR(-s_{n-2}+s_{n-3}-\nu_n)  \GR(s_{n-1}+s_n-u_{n-3}+\sigma) 
    }{\GR(s_{n-2} + s_{n-1}+s_n -u_{n-3}  -u_{n-2}-u_{n-1} +2\sigma+\nu_n ) }
\\
& \times    
 \GR(s_{n-1}  -u_{n-2} ) 
 \GR(s_{n-1}-u_{n-1}+\sigma) 
   \GR(s_n-u_{n-2}+\sigma) 
  \GR(s_n -u_{n-1} ) 
\\
& \times  \prod_{i=1}^{n-2} \GR(t_i-s_i) \prod_{i=1}^{n-2}  \GR( t_i-s_{i-1}-\nu_n)  
   \prod_{i=1}^{n-3} \frac{ \GR(t_{i}-s_{i-1} +\sigma) \GR(t_{i+1}-s_{i-1}+\sigma-\nu_n) }{  \GR(t_{i}+t_{i+1}-s_{i-1}-s_{i}+\sigma-\nu_n)}
\\
& \times \frac{ \GR(t_{n-2}-s_{n-1}-s_n) \GR(t_{n-1}+t_n-s_{n-2}-\nu_n)   }{
   \GR(  t_{n-2}+t_{n-1}+t_n-s_{n-2}-s_{n-1}-s_n-\nu_n) }
\\
& \times  
    \GR(t_{n-1}-s_{n-1}+\tfrac{\nu_n}{2})  \GR(t_{n-1}-s_n-\tfrac{\nu_n}{2}) 
    \GR(t_n-s_{n-1}-\tfrac{\nu_n}{2}) \GR(t_{n}-s_{n}+\tfrac{\nu_n}{2}).
\end{align*}

Now if $ n =4 $,  skip Step 6 and proceed directly to Step 7.
Otherwise, apply:

\noindent
{\it Step 6}. 
We apply Lemma \ref{lem:int3} to the integrals in $ s_{n-3}, s_{n-4}, \ldots, {s_2} $ successively, with 
\begin{align*}
 (a,b,c,d) & = ( -u_{i-1}+\nu_n,   -u_i+\sigma+\nu_n, -s_{i+1}-\nu_n,   s_{i-1}-u_{i-1}-u_i+\sigma+\nu_n), \\
 (\alpha,\beta, \gamma, \delta) & = (t_i,  t_{i+1}-\nu_n,   s_{i+1}+\sigma,  t_i+t_{i+1}-s_{i-1} +\sigma-\nu_n)
\end{align*}for the integral in $s_i$ ($2\le i\le n-3$).
Then the integral kernel $ K_6(\bs, \bu, \bt; \sigma) $ is 
\begin{align*}
 &K_6(\bs, \bu, \bt; \sigma)=  \GR(s_{2}-s_{1} +\sigma)  \GR(-s_{2}+s_{1}-\nu_n) \GR(-u_1+\sigma) \\
& \times  \prod_{i=1}^{n-2} \GR(s_i-u_{i-1}+\nu_{n}) 
     \prod_{i=1}^{n-4}  \GR(s_{i+2} -u_i  + \sigma +\nu_n) \\
& \times     \prod_{i=2}^{n-3} \GR(s_{i+1}-u_{i} +2\sigma+\nu_n)
    \prod_{i=2}^{n-3} \frac{\GR(s_{i}-u_{i}+\sigma+\nu_n) }{ \GR(s_{i}+s_{i+1} -u_{i-1}-u_{i}  +2\sigma + 2\nu_n ) }
     \\
&  \times \frac{ \GR(s_{n-2}-u_{n-2}-u_{n-1} +\sigma+\nu_n) \GR(s_{n-1}+s_n-u_{n-3}+\sigma)
    }{\GR(s_{n-2} + s_{n-1}+s_n -u_{n-3}  -u_{n-2}-u_{n-1} +2\sigma+\nu_n ) }\\
& \times    
 \GR(s_{n-1}  -u_{n-2} )  \GR(s_{n-1}-u_{n-1}+\sigma ) 
   \GR(s_n-u_{n-2}+\sigma) \GR(s_n -u_{n-1} ) 
\\
& \times  \prod_{i=1}^{n-2} \GR(t_i-s_i)  \prod_{i=1}^{n-2}  \GR( t_i-s_{i-1}-\nu_n)  
       \prod_{i=3}^{n-2}  \GR(t_{i}-s_{i}-2\nu_n) 
\\
& \times   \frac{ \GR(t_{1}+\sigma) \GR(t_{2}+\sigma-\nu_n)}{ \GR(t_{1}+t_{2}-s_{1}+\sigma-\nu_n)}
  \prod_{i=2}^{n-3} \frac{\GR(t_{i}-s_{i+1}-\nu_n)}{  \GR(t_{i} +t_{i+1} -s_{i} -s_{i+1}-2\nu_n)  }
\\
& \times \frac{ \GR(t_{n-2}-s_{n-1}-s_n) \GR(t_{n-1}+t_n-s_{n-2}-\nu_n)   }{
   \GR(  t_{n-2}+t_{n-1}+t_n-s_{n-2}-s_{n-1}-s_n-\nu_n) } 
\\
& \times \GR(t_{n-1}-s_{n-1}+\tfrac{\nu_n}{2}) \GR(t_{n-1}-s_n-\tfrac{\nu_n}{2}) \GR(t_n-s_{n-1}-\tfrac{\nu_n}{2}) 
    \GR(t_{n}-s_{n}+\tfrac{\nu_n}{2}).
\end{align*}

\medskip

\noindent 
{\it Step 7}. 
After replacing $s_1$ by $-s_1$, we use Barnes' second lemma:
\begin{align*}
& I_1(t_2-\nu_n, s_2+\sigma, t_1;   -s_2-\nu_n,  \nu_n ;  t_1+t_2+\sigma-\nu_n)
\\
& = \frac{ \GR(t_2) \GR(t_2-s_2-2\nu_n) \GR(t_1+\nu_n) \GR(s_2+\sigma+\nu_n) \GR(t_1-s_2-\nu_n) \GR(\sigma-\nu_n)}
{ \GR(t_1+\sigma) \GR(t_2+\sigma-\nu_n) \GR( t_1+t_2-s_2-\nu_n) }
\end{align*}
if $ n\ge4 $, and 
\begin{align*}
& I_1(t_2+t_3-\nu_n,  s_2+s_3+\sigma-\nu_n, t_1;     \nu_n, -s_2-s_3;  t_1+t_2+t_3+\sigma-\nu_n)
\\
& = \frac{ \GR(t_2+t_3) \GR(t_2+t_3-s_2-s_3-\nu_n) \GR(s_2+s_3+\sigma) \GR(\sigma-\nu_n) }
{ \GR(t_1+\sigma) \GR(t_1+t_2+t_3-s_2-s_3) \GR( t_2+t_3+\sigma-\nu_n) }
\\
& \times \GR(t_1+\nu_n) \GR(t_1-s_2-s_3)
\end{align*}
if $n=3 $.
Substituting  
$$ 
  s_i \to v_{i-1} -\nu_n \  (2 \le i \le n-2), \quad 
  s_{n-1} \to v_{n-2},  \quad 
  s_n \to v_{n-1},
$$
we then complete the proof of Lemma \ref{lem:keylemma_Pieri_SO2n}, and  consequently of Proposition \ref{prop:Pieri_SO2n}.
\epf

\section{Archimedean zeta integrals on $ {\rm SO}_{2n} \times {\rm GL}_n $} \label{sec:Arch_D}

\subsection{The archimedean zeta integral $I_{\nu,\mu}^n(s) $} 

We explain how the integral considered below arises from Method B of
Gelbart and Piatetski-Shapiro \cite[Part B]{GePS}.
We make the embedding, the representative of the open orbit, and the
normalizations explicit, retaining the Whittaker characters fixed in
Section \ref{sec:Wh}.

Put $G={\rm SO}_{2n}(\R)$, with the realization used in Section \ref{sec:Wh},
and let
\[
 H = \{
 h\in  {\rm SL}_{2n+1}(\R) \mid
 {}^t h\mathcal J h=\mathcal J \},  \qquad  
 \mathcal J=
 \begin{pmatrix}
  0&0&J_n\\
  0&-1&0\\
  J_n&0&0
 \end{pmatrix}.
\]
Thus $H$ is a split group of type $B_n$.
We use the ordered basis
$ \{  e_1,\ldots,e_n,e_0,f_n,\ldots,f_1 \}, $
and embed $G$ in $H$ by letting it act on
$\langle e_1,\ldots,e_n,f_n,\ldots,f_1\rangle$
as in Section \ref{sec:Wh}, and letting it fix $e_0$.
Denote this embedding by $\iota$.
In particular,
\[
 \iota(\beta[a_1,\ldots,a_n])
 =
 \operatorname{diag}
 (a_1,\ldots,a_n,1,a_n^{-1},\ldots,a_1^{-1}).
\]
For $K_H=H\cap {\rm O}(2n+1)$ and $K_G=G\cap {\rm O}(2n)$, we have
$
 \iota(K_G)\subset K_H.
$

Let $P=M_PN_P$ be the Siegel parabolic subgroup of $H$ stabilizing
$\langle e_1,\ldots,e_n\rangle$.
Its Levi subgroup $ M_P $ is identified with ${\rm GL}_n(\R)$ by
\[
 m(b)=
 \operatorname{diag}
 \bigl(b,1,J_n\,{}^t b^{-1}J_n\bigr),
 \qquad b\in {\rm GL}_n(\R).
\]
The modulus character $ \delta_P $ of  $P $ is given by 
\[
 \delta_P(m(b))=|\det b\,|^n.
\]
Let 
\[
  \eta  =  \operatorname{diag} (1_{n-1}, \eta_0, 1_{n-1}), \qquad 
 \eta_0=
 \begin{pmatrix}
  \frac14&\frac{1}{2\sqrt2}&\frac14\\[2pt]
  -\frac{1}{\sqrt2}&0&\frac{1}{\sqrt2}\\[2pt]
  1&-\sqrt2&1
 \end{pmatrix}.
\]
A direct calculation shows that $ \eta \in H $ and 
\[
 \eta^{-1}\langle e_1,\ldots,e_n\rangle
 =
 \langle
 e_1,\ldots,e_{n-1},e_n+\sqrt2e_0+f_n
 \rangle.
\]
This maximal isotropic subspace is not contained in $e_0^\perp$.
Thus $P\eta\iota(G)$ is the open orbit in $P\backslash H$,
namely the main orbit of \cite[Part B, Proposition 3.1.2]{GePS}.

Write $U_D$ and $U_A$ for the upper triangular maximal unipotent
subgroups of $G$ and ${\rm GL}_n(\R)$, respectively, and
$\psi_D$ and $\psi_A$ for the characters defined in Section \ref{sec:Wh}.
For every $u=(u_{ij})\in U_D$, we have
\[
 \eta\iota(u)\eta^{-1}=m(z(u))v(u),
 \qquad
 z(u)\in U_A,\ v(u)\in N_P.
\]
Indeed, $U_D$ preserves the full flag ending in the displayed
maximal isotropic subspace.
Its action on the ordered basis
$ \{ e_1,\ldots,e_{n-1},e_n+\sqrt2e_0+f_n \} $
gives
\begin{align*}
 z(u)_{i,i+1} & = 
  \begin{cases} 
  u_{i,i+1} & \text{if $ 1\leq i\leq n-2$},
\\
  u_{n-1,n}+u_{n-1,n+1} & \text{if $ i=n-1$}. \end{cases}
\end{align*}
In particular,
\begin{align}
 \label{eq:SOGL-character-compatibility}
 \psi_A(z(u))=\psi_D(u).
\end{align}

Let $\pi_\mu$ denote the spherical principal series
representation of ${\rm GL}_n(\R)$ with parameter $\mu$.
We use its $\psi_A^{-1}$-Whittaker model for the inducing data.
More precisely, set
\[
 \epsilon_n=
 \operatorname{diag}
 \bigl((-1)^{n-1},(-1)^{n-2},\ldots,-1,1\bigr),
 \qquad
 W^{A,-}_{n,\mu}(b)
 =
 W^A_{n,\mu}(\epsilon_n b), \quad b \in {\rm GL}_n(\R).
\]
Then $W^{A,-}_{n,\mu}$ is a right ${\rm O}(n)$-invariant $\psi_A^{-1}$-Whittaker function
satisfying
\[
 W^{A,-}_{n,\mu}(\alpha[a_1,\ldots,a_n])
 = W^A_{n,\mu}(\alpha[a_1,\ldots,a_n]).
\]

Let $F_{\mu,s}$ be the scalar-valued realization, in this
Whittaker model, of the $K_H$-fixed section of the normalized
induced representation
$ \operatorname{Ind}_P^H
 (\pi_\mu\otimes|\det|^{s-\frac12}). $
We normalize it by
\begin{align}
 \label{eq:SOGL-spherical-section}
 F_{\mu,s}(m(b)vk)
 &=
 |\det b \,|^{s+\frac{n-1}{2}}
 W^{A,-}_{n,\mu}(b),
 \qquad  b\in {\rm GL}_n(\R),\ v\in N_P, \ k\in K_H.
\end{align}
By \eqref{eq:SOGL-character-compatibility}, we have 
\[
 F_{\mu,s}(\eta\iota(ug))
 =
 \psi_D(u)^{-1}  F_{\mu,s}(\eta\iota(g)),
 \qquad
 u\in U_D,\ g\in G.
\]

The real local Rankin-Selberg integral supplied by Method B of \cite[Part B, Theorem 3.B and Section 5.1]{GePS},
with these choices of data, is therefore
\begin{equation}
 \label{eq:SOGL-local-zeta-integral}
 Z(s,W^D_{n,\nu},F_{\mu,s})
 =
 \int_{U_D\backslash G}
 W^D_{n,\nu}(g)
 F_{\mu,s}(\eta\iota(g))\,dg.
\end{equation}
The integrand is left $U_D$-invariant and right $K_G$-invariant.
We normalize the quotient measure so that its Iwasawa
expression is
\[
 dg=
 \prod_{i=1}^n a_i^{-2(n-i)} \frac{da_i}{a_i}\,dk,
 \qquad
 a=\beta[a_1,\ldots,a_n],
 \qquad
 \int_{K_G}dk=1.
\]

Let us compute the $PK_H$-decomposition of
$\eta\iota(a)$.
This reduces to a three-dimensional calculation.
For $t>0$, put
\[
 r_t=\frac{1-t^2}{\sqrt2(1+t^2)},
 \qquad
 x(t)=
 \begin{pmatrix}
  1&r_t& \frac12 r_t^2 \\
  0&1&r_t\\
  0&0&1
 \end{pmatrix},
\]
and
\[
 k(t)=
 \frac{1}{1+t^2}
 \begin{pmatrix}
  1&\sqrt2t&t^2\\
  -\sqrt2t&1-t^2&\sqrt2t\\
  t^2&-\sqrt2t&1
 \end{pmatrix} \in {\rm SO}(3).
\]
Note that $k(t)$ preserves the restriction of $\mathcal J$ to $\langle e_n,e_0,f_n\rangle$.
Direct computation gives
\begin{align}
 \label{eq:SOGL-rank-one-decomposition}
 \eta_0
 \begin{pmatrix}
  t&0&0\\
  0&1&0\\
  0&0&t^{-1}
 \end{pmatrix}
 =
 x(t)
 \begin{pmatrix}
  (t+t^{-1})^{-1}&0&0\\
  0&1&0\\
  0&0&t+t^{-1}
 \end{pmatrix}
 k(t).
\end{align}
It follows that, for $ a = \beta[a_1,\dotsc,a_n] $, we have
\[
 \eta\iota(a)=v(a)m(b(a))k(a),
 \qquad
 v(a)\in N_P,\ k(a)\in K_H,
\]
where
\[
 b(a)=
 \alpha [
 a_1,\ldots,a_{n-1},(a_n+a_n^{-1})^{-1} ].
\]
Hence \eqref{eq:SOGL-spherical-section} yields
\[
 F_{\mu,s}(\eta\iota(a))
 =
 \Bigl(
 \frac{a_1\cdots a_{n-1}}{a_n+a_n^{-1}}
 \Bigr)^{s+\frac{n-1}{2}}
 W^A_{n,\mu}(b(a)).
\]
Substitution in \eqref{eq:SOGL-local-zeta-integral} gives
\[
 \begin{aligned}
 Z(s,W^D_{n,\nu},F_{\mu,s})
 ={}&
 \int_{(\R_+)^n}  W^D_{n,\nu}(a) W^A_{n,\mu}(b(a)) 
 \cdot  \Bigl(
 \frac{a_1\cdots a_{n-1}}{a_n+a_n^{-1}}
 \Bigr)^{\!s+\frac{n-1}{2}}
 \prod_{i=1}^n a_i^{-2(n-i)}\, \frac{da_i}{a_i}.
 \end{aligned}
\]
Consequently, we get
\begin{align*}
 Z(s,W^D_{n,\nu},F_{\mu,s})
 & = 
 \int_{(\R_{+})^n}
 \widetilde W^D_{n,\nu}
 \bigl(\beta[a_1,\ldots,a_n]\bigr)
 \\
 &\times
 \widetilde W^A_{n,\mu}
 \left(
 \alpha [a_1,\ldots,a_{n-1},(a_n+a_n^{-1})^{-1} ]
 \right) 
 \cdot   \Bigl( \frac{a_1\cdots a_{n-1}}{a_n+a_n^{-1}} \Bigr)^{\!s}
\, \prod_{i=1}^n\frac{da_i}{a_i}.
\end{align*}
We denote this integral by $I^n_{\nu,\mu}(s)$.
Thus $I^n_{\nu,\mu}(s)$ is precisely the real local
Rankin-Selberg integral for the embedding, open-orbit
representative, spherical data, and Haar measure specified above.



\begin{prop} \label{prop:I}
We have 
\begin{align*}
   I_{\nu, \mu}^n(s) 
& = \frac{2^{-1} }{(2\pi \sqrt{-1})^n}  \int_{s_1,\dotsc,s_n} 
      V_{n,\nu}^D(s_1, \dotsc, s_n) \\
& \times V_{n, \mu}^A(s-s_1, 2s-s_2,\dotsc, (n-2)s-s_{n-2}, (n-1)s-s_{n-1}-s_n) 
\\
& \times  
 \frac{ \GR(ns-2s_{n-1}+|\mu|) \GR(ns -2s_n+|\mu|)}{ \GR(2ns-2s_{n-1}-2s_n+2|\mu|)}   \,\prod_{i=1}^n ds_i.
\end{align*}
\end{prop}

\bpf
Inserting the Mellin-Barnes representation of $ \widetilde{W}_{n, \mu}^A $, we get
\begin{align*}
 I_{\nu, \mu}^n(s) 
& =\frac{1}{(2\pi \sqrt{-1})^{n-1}}  \int_{s_1,\dotsc,s_{n-1} } V_{n, \mu}^A(s_1,\dotsc,s_{n-1})
   \int_{(\R_+)^{n}} \widetilde{W}_{n,\nu}^D ( \beta[a_1,\dotsc, a_n] )  
 \\
& \times 
      a_1^{s-s_1} a_2^{s+s_1-s_2} \cdots a_{n-1}^{s+s_{n-2}-s_{n-1} } (a_n+a_n^{-1})^{-s-s_{n-1}-|\mu|} 
    \prod_{i=1}^{n}  \frac{da_i}{a_i} \prod_{i=1}^{n-1} ds_i.
\end{align*}
Since 
\[
  \int_0^{\infty} (a+a^{-1})^{-p} a^s \frac{da}{a} = \frac{ \GR(p-s) \GR(p+s)}{ 2 \GR(2p) } 
\]
for $ {\rm Re}(p \pm s) > 0 $, we have
\begin{align*}
& (a_n+a_n^{-1})^{-s-s_{n-1}-|\mu|} 
\\
& = \frac{1}{4\pi \sqrt{-1}} \int_t 
  \frac{ \GR(s+s_{n-1}+|\mu|-t) \GR(s+s_{n-1} +|\mu|+t)}{ \GR(2s+2s_{n-1}+2|\mu|)}
     \,a_n^{-t} \,dt
\\
& = \frac{1 }{2 \pi \sqrt{-1}} \int_t
  \frac{ \GR(2s-2t+|\mu|) \GR(2s_{n-1}+2t+|\mu|)}
   { \GR(2s+2s_{n-1}+2|\mu|)}  \, a_n^{-2t-s_{n-1}+s} \,dt.
\end{align*}
Then we have 
\begin{align*}
&   I_{\nu, \mu}^n(s) \\
& = \frac{1 }{(2\pi \sqrt{-1})^n} \int_{t,s_1,\dots,s_{n-1}}  V_{n, \mu}^A(s_1,\dotsc, s_{n-1}) 
    \frac{ \GR(2s-2t+|\mu|) \GR(2s_{n-1}+2t+|\mu|)}{ \GR(2s+2s_{n-1}+2|\mu|)}  
\\
& \times 
    \int_{(\R_+)^n} 
       \widetilde{W}_{n,\nu}^D ( \beta[a_1,\dotsc, a_n] )
       \prod_{i=1}^{n-2} \Bigl( \frac{a_i}{a_{i+1}} \Bigr)^{\! is-s_i} \cdot 
       \Bigl( \frac{a_{n-1}}{a_n} \Bigr)^{\! \frac{n-2}{2}s+t} (a_{n-1}a_n)^{\frac{n}{2}s-s_{n-1}-t}  
   \, \prod_{i=1}^{n} \frac{da_i}{a_i} \\
& \times ds_1\cdots ds_{n-1}  \,dt
\\
& = \frac{2^{-1} }{ (2\pi \sqrt{-1})^n} \int_{s_1,\dotsc,s_{n-1},t}  
    V_{n,\nu}^D(s-s_1,2s-s_2, \dotsc, (n-2)s-s_{n-2},  \tfrac{n-2}{2}s+t, \tfrac{n}{2}s-s_{n-1}-t) 
\\
& \times  V_{n, \mu}^A (s_1,\dotsc, s_{n-1})  
    \frac{ \GR(2s-2t+|\mu|) \GR(2s_{n-1}+2t+|\mu|)}{ \GR(2s+2s_{n-1}+2|\mu|)}   \,ds_1\cdots ds_{n-1}\, dt.
\end{align*}
We substitute $ s_i \to is -s_i $ $(1\le i \le n-2) $, $t \to s_{n-1}-\frac{n-2}{2}s$, and $s_{n-1} \to (n-1)s-s_n-s_{n-1} $, to get 
\begin{align*}
   I_{\nu, \mu}^n(s) 
& = \frac{2^{-1} }{(2\pi \sqrt{-1})^n}  \int_{s_1,\dotsc,s_n} 
      V_{n,\nu}^D(s_1, \dotsc, s_n) \\
& \times V_{n, \mu}^A (s-s_1, 2s-s_2,\dotsc, (n-2)s-s_{n-2}, (n-1)s-s_{n-1}-s_n) 
\\
& \times  
 \frac{ \GR(ns-2s_{n-1}+|\mu|) \GR(ns -2s_n+|\mu|)}{ \GR(2ns-2s_{n-1}-2s_n+2|\mu|)}   \,\prod_{i=1}^n ds_i,
\end{align*}
as desired.
\epf

\subsection{Computation of $I_{\nu,\mu}^n(s)$}   \label{subsec:SOGL_zeta}
This subsection is the main part of this paper. Our aim in this subsection is to prove the following.

\begin{thm} \label{thm:SOGL}
For $ n \ge 2 $, $ \nu=(\nu_1,\dotsc,\nu_n) \in \C^n $ and $ \mu = (\mu_1,\dotsc,\mu_n) \in \C^n $, we have 
\begin{align*}
 I_{\nu,\mu}^n(s)  = \prod_{1 \le i \le j \le n} \frac{1}{\GR(2s+\mu_i+\mu_j)}
  \prod_{1 \le i, j \le n} \{ \GR(s+\nu_i+\mu_j) \GR(s-\nu_i+\mu_j) \}.
\end{align*}
\end{thm}

As in our previous paper \cite{IS13}, our argument is by induction on $n$.
The case of $n=2$ is nothing but the archimedean zeta integral for the triple $L$-function on ${\rm GL}_2$, 
which was evaluated by Ikeda \cite{Ik99}. Here, we present a very short proof of Ikeda's result.
To complete the proof of Theorem \ref{thm:SOGL}, we then carry out our induction argument, in \S \ref{subsec:SOGL_induction}.

\begin{prop} \label{prop:SOGL_induction}
{\rm(i)} $($\cite{Ik99}$)$ Theorem \ref{thm:SOGL} is true for $n=2$.
\medskip \\
{\rm (ii)} Let $ n \ge 3 $.
If we write $ \tilde{\nu}= ( \nu_1,\dotsc,\nu_{n-1}) $ and $\tilde{\mu}= (\mu_1, \dotsc, \mu_{n-1}) $, then we have
\begin{align} \label{eqn:SO_zeta_rec}
\begin{split}
   I_{\nu,\mu}^n(s) 
 & =  \prod_{i=1}^n \frac{ \GR(s+\nu_n+\mu_i) \GR(s-\nu_n+\mu_i) }{  \GR(2s+\mu_i+\mu_n) }  
\\
 & \times 
   \prod_{i=1}^{n-1} \{ \GR(s+\nu_i+\mu_n)  \GR(s-\nu_i+\mu_n) \} \cdot 
  I_{\tilde{\nu}, \tilde{\mu}}^{n-1}(s). 
\end{split}
\end{align}
\end{prop}

Let us prove Proposition \ref{prop:SOGL_induction}(i); the proof of (ii) is given in the next subsection.

By the explicit formulas for $ V_{2,(\nu_1,\nu_2)}^D(s_1, s_2) $ and $V_{2,(\mu_1,\mu_2)}^A(s_1)$ (see Propositions \ref {prop:exp_SO2n} and \ref{prop:exp_GLn}), we have
\begin{align*}
 I_{(\nu_1,\nu_2), (\mu_1,\mu_2)}^2 (s)
& = \frac{1}{(4\pi \sqrt{-1})^2} \int_{s_1,s_2}
     \GR(s_1 +\tfrac{\nu_1-\nu_2}{2}) \GR(s_1-\tfrac{\nu_1-\nu_2}{2})
\\
& \times  \GR(s_2+\tfrac{\nu_1+\nu_2}{2}) \GR(s_2-\tfrac{\nu_1+\nu_2}{2})\GR(s-s_1-s_2+\mu_1) \GR(s-s_1-s_2+\mu_2) \\
& \times 
 \frac{ \GR(2s-2s_1+\mu_1+\mu_2) \GR(2s-2s_2+\mu_1+\mu_2)}{ \GR(4s-2s_1-2s_2+2\mu_1+2\mu_2)} \,ds_1 ds_2.
\end{align*} 
After replacing $ s_2 $ by $ -s_2 $, we apply Lemma \ref{lem:int2}, with 
\begin{align*}
 (a,b,c) &  
 = (s-s_1+\mu_1, s-s_1+\mu_2, 2s+\mu_1+\mu_2),  \\
 (d,e,f) & 
 = ( \tfrac{\nu_1+\nu_2}{2}, -\tfrac{\nu_1+\nu_2}{2} , 4s+2\mu_1+2\mu_2-2s_1),
\end{align*}to the integral in $s_2$,
to obtain
\begin{align*}
 I_{(\nu_1,\nu_2), (\mu_1,\mu_2)}^2 (s)
& = \frac{1}{ \GR(2s+\mu_1+\mu_2)  } \cdot \frac{1}{(4\pi \sqrt{-1})^2 } \int_{s_1,s_2} 
   \GR(s_1+ \tfrac{\nu_1-\nu_2}{2}) \GR(s_1 -\tfrac{\nu_1-\nu_2}{2}) 
\\
&  \times  \GR(s-s_1+\tfrac{\nu_1+\nu_2}{2}+\mu_1) \GR(s-s_1-\tfrac{\nu_1+\nu_2}{2}+\mu_1) \GR(s_2+ s-s_1+\mu_2) 
   \\
& \times
   \frac{  \GR(s_2+2s+\mu_1+\mu_2+\tfrac{\nu_1+\nu_2}{2}) \GR(s_2+2s+\mu_1+\mu_2-\tfrac{\nu_1+\nu_2}{2})}
 { \GR(s_2+3s-s_1+2\mu_1+\mu_2) } 
\\
& \times \GR(-s_2+\tfrac{\nu_1+\nu_2}{2}) \GR(-s_2-\tfrac{\nu_1+\nu_2}{2}) \,ds_1 ds_2.
\end{align*}
Using Barnes' second lemma for the integral in $s_1$, we get
\begin{align*}
  I_{(\nu_1,\nu_2), (\mu_1,\mu_2)}^2 (s)
& = \frac{ \GR(s+\nu_1+\mu_1) \GR(s-\nu_1+\mu_1) \GR(s+\nu_2+\mu_1) \GR(s-\nu_2+\mu_1) }
  {  \GR(2s+\mu_1+\mu_2) \GR(2s+2\mu_1) } 
\\
& \times 
\frac{1}{4\pi \sqrt{-1}} \int_{s_2} 
  \GR(s+s_2+\tfrac{\nu_1-\nu_2}{2}+\mu_2) \GR(s+s_2-\tfrac{\nu_1-\nu_2}{2}+\mu_2)
\\
& \times 
  \GR(-s_2+\tfrac{\nu_1+\nu_2}{2}) \GR(-s_2-\tfrac{\nu_1+\nu_2}{2}) \,ds_2.
\end{align*}
Then Barnes' first lemma, applied to this remaining integral, implies the assertion.

\subsection{Proof of Proposition \ref{prop:SOGL_induction}(ii)} \label{subsec:SOGL_induction}
In this subsection, we prove the second statement of Proposition \ref{prop:SOGL_induction}. 

We set
$ \bs = (s_1,\dotsc,s_n) \in \C^n $,
$ \bt = (t_1,\dotsc, t_{n-1}) \in \C^{n-1} $, and 
$ \bq = (q_1,\dotsc, q_{n-2}) \in \C^{n-2} $.
By Propositions \ref{prop:exp_GLn} and \ref{prop:exp_SO2n}, we have 
\begin{align*}
   I_{\nu, \mu}^n (s) 
& = \frac{2^{n-2}}{(4\pi \sqrt{-1})^{3n-3}}  \int_{s_1,\dotsc,s_n, t_1,\dotsc, t_{n-1}, q_1,\dotsc,q_{n-2} } 
      V_{n-1, \tilde{\nu}}^D(\bt) \\
& \times V_{n-1, (\mu_1,\dotsc, \mu_{n-2}, \mu_n) }^A(\bq ) 
   P(\bs,\bt, \bq; s) \prod_{i=1}^n ds_i \prod_{i=1}^{n-1} dt_i  \prod_{i=1}^{n-2} dq_i,
\end{align*}
where
\begin{align*}
   P(\bs,\bt, \bq; s)
& = L_{n,\nu_n}(\bs,\bt) 
  \prod_{i=1}^{n-2} \GR(is-s_i-q_i ) 
   \prod_{i=1}^{n-2} \GR(is-s_{i}-q_{i-1} + \mu_{n-1}) 
\\
& \times 
  \GR((n-1)s-s_{n-1} -s_n-q_{n-2} +\mu_{n-1}) \\
& \times \GR( (n-1)s-s_{n-1}-s_n + |\mu|-\mu_{n-1}) 
\\
& \times  
 \frac{ \GR(ns-2s_{n-1}+|\mu|) \GR(ns -2s_n+|\mu|)}{ \GR(2ns-2s_{n-1}-2s_n+2|\mu|)}.
\end{align*}
Here we understand that $ q_0= 0$.

\medskip 

\noindent 
{\it Step 1.}
For  $1\le i\le n-3$, we apply Lemma \ref{lem:int1}(ii) to the integral in $s_i$, with 
$$
  (a,b,c, d,e,  f)
 = (-s_{i-1}+\nu_n,  -t_{i-1}-\nu_n,  -t_i,   is-q_i , s+s_{i-1}+\mu_{n-1},  s_{i+1}-t_{i-1}-t_i),
$$
and also substitute $ s_i \to s_i+s_{i-1}-s_{i+1} +s +\mu_{n-1}-\nu_n$ for each $i$ in succession.

Next:  if $ n=3 $, skip the following computation,  and proceed to    the treatment, below, of the integral in $s_{n-2}$. If $n>3$ then, since
\begin{align*}
&  \GR((i+1)s-s_{i+1}-q_i+\mu_{n-1}) \GR(s_{i+1}-t_i+\nu_n) \GR(s_{i+1}-t_{i-1}) 
\\
& \times 
\int_{s_i} 
 \frac{ \GR(-s_{i-1}+s_i+\nu_n) \GR(s_i-t_{i-1}-\nu_n) \GR(s_i-t_i) }{ \GR(s_i+s_{i+1}-t_{i-1}-t_i) }\\
 & \times \GR(is-s_i-q_i) \GR(s+s_{i-1}-s_i+\mu_{n-1}) \,ds_i
\\
& = \GR(is-s_{i-1}-q_i + \nu_n) \GR(s+s_{i-1}-t_{i-1}+\mu_{n-1}-\nu_n) \GR(s+s_{i-1}-t_i+\mu_{n-1})
\\
& \times 
 \int_{s_i} \frac{ \GR(s+s_i-s_{i+1}+\mu_{n-1}) \GR(s_i-t_i) \GR(s_i-t_{i-1}-\nu_n) }{ \GR(s+s_{i-1}+s_i-t_{i-1}-t_i+\mu_{n-1}-\nu_n) }\\
 & \times \GR(is-s_i-q_i) \GR(-s_i+s_{i+1}+\nu_n)  \,ds_i,
\end{align*}
we get 
\begin{align*}
   I_{\nu, \mu}^n(s) 
& = \GR(s+\mu_{n-1}-\nu_n) \cdot \frac{2^{n-2}}{(4\pi \sqrt{-1})^{3n-3}}  
    \int_{s_1,\dotsc,s_n, t_1,\dotsc, t_{n-1}, q_1,\dotsc,q_{n-2} } 
      V_{n-1, \tilde{\nu}}^D(\bt) \\
& \times V_{n-1, (\mu_1,\dotsc,\mu_{n-2},\mu_n) }^A(\bq) 
\\
& \times \prod_{i=1}^{n-3} \GR(s_i-t_i) 
   \prod_{i=1}^{n-4} \frac{1}{\GR(s+s_i+s_{i+1}-t_i-t_{i+1} +\mu_{n-1}-\nu_n)  } 
\\
& \times 
   \prod_{i=1}^{n-3} \GR(s+s_{i-1}-t_i+\mu_{n-1}) 
   \prod_{i=1}^{n-2}  \GR(s_i-t_{i-1}-\nu_n) 
\\
& \times 
   \frac{ \GR(s_{n-2}-t_{n-2}-t_{n-1}) \GR(s_{n-1}+s_n-t_{n-3}) \GR(s_{n-1}-t_{n-2} + \frac{\nu_n}{2}) }
  {\GR(s_{n-2}+s_{n-1}+s_n-t_{n-3}-t_{n-2}-t_{n-1}) } 
\\
& \times \GR(s_{n-1}-t_{n-1}-\tfrac{\nu_n}{2}) \GR(s_{n}-t_{n-1}+\tfrac{\nu_n}{2}) \GR(s_{n}-t_{n-2}-\tfrac{\nu_n}{2})
\\  
& \times 
  \prod_{i=1}^{n-2} \GR(is-s_i-q_i ) 
   \prod_{i=1}^{n-3} \GR(is-s_{i-1}-q_i + \nu_n) 
\\
& \times \prod_{i=1}^{n-4} \GR(s+s_i-t_i + \mu_{n-1}-\nu_n) \cdot 
    \frac{ \GR(s+s_{n-3}-s_{n-2}+\mu_{n-1})  \GR(-s_{n-3}+s_{n-2}+\nu_n)}{ \GR(s+s_1-t_1+\mu_{n-1}-\nu_n)}
\\
& \times 
 \GR( (n-1)s-s_{n-1}-s_n + |\mu|-\mu_{n-1}) \GR((n-1)s-s_{n-1} -s_n-q_{n-2} +\mu_{n-1}) 
\\
& \times  
 \frac{ \GR(ns-2s_{n-1}+|\mu|) \GR(ns -2s_n+|\mu|)}{ \GR(2ns-2s_{n-1}-2s_n+2|\mu|)}   
  \,\prod_{i=1}^n ds_i \prod_{i=1}^{n-1} dt_i  \prod_{i=1}^{n-2} dq_i.
\end{align*}
Here we understand that $ s_0 =t_0 =0 $.

We now apply Lemma \ref{lem:int1}(ii), with  
\begin{align*}
  (a,b,c)
& = (-s_{n-3}+\nu_n,  -t_{n-3}-\nu_n,  -t_{n-2}-t_{n-1}),   \\
  (d,e, f)
& = ( (n-2)s-q_{n-2},  s+s_{n-3}+\mu_{n-1},  s_{n-1}+s_n-t_{n-3}-t_{n-2}-t_{n-1})
\end{align*}to the integral in $s_{n-2}$,
and then substitute $ s_{n-2} \to s_{n-2}+s_{n-3}-s_{n-1}-s_n+s+\mu_{n-1}-\nu_n $.
We get 
\begin{align*}
   I_{\nu, \mu}^n(s) 
& = \GR(s-\nu_n+\mu_{n-1}) \cdot \frac{2^{n-2}}{(4\pi \sqrt{-1})^{3n-3}}  
    \int_{s_1,\dotsc,s_n, t_1,\dotsc, t_{n-1}, q_1,\dotsc,q_{n-2} } 
      V_{n-1, \tilde{\nu}}^D(\bt) \\
& \times V_{n-1, (\mu_1,\dotsc,\mu_{n-2},\mu_n)  }^A(\bq) 
\\
& \times \prod_{i=1}^{n-3} \GR(s_i-t_i) 
    \prod_{i=1}^{n-4} \frac{1}{ \GR(s+s_i+s_{i+1}-t_i-t_{i+1} +\mu_{n-1}-\nu_n)  } 
\\
& \times 
   \prod_{i=1}^{n-3} \GR(s+s_{i-1}-t_i+\mu_{n-1}) \cdot \GR(s+s_{n-3}-t_{n-2}-t_{n-1}+\mu_{n-1}) 
\\
& \times 
   \prod_{i=1}^{n-2}  \GR(s_i-t_{i-1}-\nu_n) 
\\
& \times 
   \frac{ \GR(s_{n-2}-t_{n-2}-t_{n-1})  \GR(s_{n-1}-t_{n-2} + \frac{\nu_n}{2}) }
  {\GR(s+s_{n-2}+s_{n-3}-t_{n-3}-t_{n-2}-t_{n-1} + \mu_{n-1}-\nu_n)  } 
\\
& \times \frac{ \GR(s_{n-1}-t_{n-1}-\tfrac{\nu_n}{2}) \GR(s_{n}-t_{n-1}+\tfrac{\nu_n}{2}) \GR(s_{n}-t_{n-2}-\tfrac{\nu_n}{2})}{ 
   \GR(s_{n-1}+s_n-t_{n-2}-t_{n-1}+\nu_n) } 
\\  
& \times 
  \prod_{i=1}^{n-2} \GR(is-s_i-q_i) 
   \prod_{i=1}^{n-2} \GR(is-s_{i-1}-q_i+ \nu_n) 
\\
& \times \prod_{i=2}^{n-3} \GR(s+s_i-t_i + \mu_{n-1}-\nu_n) 
  \cdot \GR(s+s_{n-2}-s_{n-1}-s_n+\mu_{n-1})  \\
& \times 
 \GR(-s_{n-2}+s_{n-1}+s_n+\nu_n)
 \GR( (n-1)s-s_{n-1}-s_n + |\mu|-\mu_{n-1}) 
\\
& \times  
 \frac{ \GR(ns-2s_{n-1}+|\mu|) \GR(ns -2s_n+|\mu|)}{ \GR(2ns-2s_{n-1}-2s_n+2|\mu|)}   
  \,\prod_{i=1}^n ds_i \prod_{i=1}^{n-1} dt_i  \prod_{i=1}^{n-2} dq_i.
\end{align*}

\noindent 
{\it Step 2}. 
We  apply Proposition \ref{prop:Pieri_GLn},  with 
\[
  s_i = is-s_i \ \ (1\le i \le n-2), \qquad  \sigma = s + \nu_n ,
\]to the multiple integral in  $ q_1,q_2,\dotsc, q_{n-2},$
to get 
\begin{align*}
   I_{\nu, \mu}^n(s) 
& = \GR(s-\nu_n+\mu_{n-1}) \prod_{i=1}^{n-2} \GR(s+\nu_n+\mu_i) \cdot \GR(s+\nu_n+\mu_n) \\
& \times \frac{2^{n-2}}{(4\pi \sqrt{-1})^{2n-1}}  
    \int_{s_1,\dotsc,s_n, t_1,\dotsc, t_{n-1}} 
      V_{n-1, \tilde{\nu}}^D( \bt) \\
& \times V_{n-1, (\mu_1,\dotsc,\mu_{n-2},\mu_n) }^A(s-s_1, 2s-s_2,\dotsc, (n-2)s-s_{n-2}) 
\\
& \times\prod_{i=1}^{n-3} \GR(s_i-t_i)
   \prod_{i=1}^{n-4} \frac{1}{ \GR(s+s_i+s_{i+1}-t_i-t_{i+1} +\mu_{n-1}-\nu_n)  } 
\\
& \times 
   \prod_{i=1}^{n-3} \GR(s+s_{i-1}-t_i+\mu_{n-1}) 
   \prod_{i=1}^{n-2}  \GR(s_i-t_{i-1}-\nu_n) 
\\
& \times 
   \frac{ \GR(s+s_{n-3}-t_{n-2}-t_{n-1}+\mu_{n-1})  \GR(s_{n-2}-t_{n-2}-t_{n-1})   }
  {\GR(s+s_{n-2}+s_{n-3}-t_{n-3}-t_{n-2}-t_{n-1} + \mu_{n-1}-\nu_n)  } 
\\
& \times 
  \frac{ \GR(s_{n-1}-t_{n-2} + \tfrac{\nu_n}{2})
   \GR(s_{n-1}-t_{n-1}-\tfrac{\nu_n}{2})  }{
   \GR(s_{n-1}+s_n-t_{n-2}-t_{n-1}+\nu_n)} 
\\  
& \times \GR(s_{n}-t_{n-1}+\tfrac{\nu_n}{2}) \GR(s_{n}-t_{n-2}-\tfrac{\nu_n}{2})
\\
& \times \prod_{i=2}^{n-3} \GR(s+s_i-t_i + \mu_{n-1}-\nu_n) 
\\
& \times 
  \GR(s+s_{n-2}-s_{n-1}-s_n+\mu_{n-1})  \GR(-s_{n-2}+s_{n-1}+s_n+\nu_n)
\\
& \times 
 \frac{\GR( (n-1)s-s_{n-1}-s_n + |\mu|-\mu_{n-1})}{ \GR( (n-1)s-s_{n-2} + |\mu| -\mu_{n-1}+\nu_n) }     
\\
& \times  
 \frac{ \GR(ns-2s_{n-1}+|\mu|) \GR(ns -2s_n+|\mu|)}{ \GR(2ns-2s_{n-1}-2s_n+2|\mu|)}   
  \,\prod_{i=1}^n ds_i \prod_{i=1}^{n-1} dt_i.
\end{align*}

Now if $n=3$, we omit Step 3 below, and proceed directly from the expression obtained in Step 2 to the application of 
Lemma \ref{lem:edge_integral} in Step 4.

Otherwise, we have 

\noindent 
{\it Step 3}. For $ n \ge 4 $, use Proposition \ref{prop:exp_GLn} for $V_{n-1, (\mu_1,\dotsc,\mu_{n-2},\mu_n) }^A$:
\begin{align*}
& V_{n-1, (\mu_1,\dotsc,\mu_{n-2},\mu_n) }^A(s-s_1, 2s-s_2, \dotsc, (n-2)s-s_{n-2})  
\\
& = \frac{2^{-1}}{ (4\pi \sqrt{-1})^{n-3} } \int_{q_1,\dotsc, q_{n-3}} V_{n-2,(\mu_1,\dotsc, \mu_{n-2})}^A(\bq') \\
& \times 
  \prod_{i=1}^{n-3} \GR(is-s_i-q_i ) \cdot \GR((n-2)s-s_{n-2}+|\mu|-\mu_{n-1}-\mu_n) \\
& \times 
  \prod_{i=1}^{n-2} \GR(is-s_i -q_{i-1}+ \mu_{n}) \prod_{i=1}^{n-3} dq_i
\end{align*}
with $ \bq' = (q_1,\dotsc, q_{n-3}) $.

Now consider the integrals  in $ s_1 $ through $s_{n-3}$.
If $ n \ge 5 $, we apply Lemma \ref{lem:int1}(ii) to the integral in $s_i$,  with \begin{align*}
  (a,b,c) & =  (-s_{i-1}-\nu_n,  -t_{i},  s-t_{i+1}+\mu_{n-1}),  \\ 
  (d,e,f)  & = ( is-q_i , s+s_{i-1}+\mu_{n},  s+s_{i+1}-t_{i}-t_{i+1}+\mu_{n-1}-\nu_n),
\end{align*}
and then substitute $ s_i \to s_i+s_{i-1}-s_{i+1} +s +\mu_n +\nu_n$, for $i=1,2,\ldots, n-4$ successively. 

We next apply the same lemma to the integral in $s_{n-3}$, with  
\begin{align*}
(a,b,c) & = ( -s_{n-4}-\nu_n,  -t_{n-3},  s-t_{n-2}-t_{n-1}+\mu_{n-1} ), \\
 (d,e,f) & =  ( (n-3)s-q_{n-3},  s+s_{n-4}+\mu_{n},  s+s_{n-2}-t_{n-3}-t_{n-2}-t_{n-1} +\mu_{n-1}-\nu_n),
\end{align*}
then substitute $ s_{n-3} \to s_{n-3}+s_{n-4}-s_{n-2}+s+\mu_{n}+\nu_n$. The result is that
\begin{align*}
   I_{\nu, \mu}^n(s) 
& = \GR(s-\nu_n+\mu_{n-1}) \prod_{i=1}^{n-2} \GR(s+\nu_n+\mu_i) \cdot \GR(s+\nu_n+\mu_n)  \\
& \times \frac{2^{n-3}}{(4\pi \sqrt{-1})^{3n-4}}  
    \int_{s_1,\dotsc,s_n, t_1,\dotsc, t_{n-1},q_1,\dotsc,q_{n-3}} 
      V_{ n-1, \tilde{\nu}}^D(\bt )  
   V_{n-2, (\mu_1,\dotsc, \mu_{n-2})}^A (\bq')
\\
& \times \prod_{i=1}^{n-3} \GR(is-s_{i-1}-q_i-\nu_n) \cdot  \GR(-s_{n-3}+s_{n-2}-\nu_n) 
\\
& \times     \prod_{i=1}^{n-3} \GR(is-s_i-q_i) \cdot 
 \GR( (n-2)s-s_{n-2} + |\mu|-\mu_{n-1}-\mu_n) 
\\
& \times 
 \prod_{i=1}^{n-3} \GR(s_i-t_i)   
  \prod_{i=1}^{n-3} \GR(s+s_{i-1}-t_i+\mu_{n})
\\
& \times      
    \prod_{i=1}^{n-4} \GR(2s+s_{i-1}-t_{i+1}+\mu_{n-1}+\mu_n) \cdot \GR(2s+s_{n-4}-t_{n-2}-t_{n-1}+\mu_{n-1}+\mu_n)
\\
& \times\frac{ \GR(s+s_{n-3}-s_{n-2}+\mu_{n})  }{  
\prod_{i=1}^{n-4} \GR(2s+s_{i-1}+s_{i} -t_i-t_{i+1} + \mu_{n-1}+\mu_n)   } 
\\
& \times 
 \frac{ 1}{  \GR(2s+s_{n-4} +s_{n-3}-t_{n-3}-t_{n-2}-t_{n-1}+\mu_{n-1}+\mu_n)}
  \\
& \times   
    \prod_{i=1}^{n-3} \GR(s+s_{i-1}-t_i+\mu_{n-1}) \cdot \GR(s+s_{n-3}-t_{n-2}-t_{n-1}+\mu_{n-1}) \\
& \times 
   \frac{  \GR(s_{n-2}-t_{n-2}-t_{n-1})   }
  {\GR(s+s_{n-2}-t_{n-2}-t_{n-1} + \mu_{n-1}-\nu_n) \GR(s_{n-1}+s_n-t_{n-2}-t_{n-1}+\nu_n) } 
\\
& \times 
   \GR(s_{n-1}-t_{n-2} + \tfrac{\nu_n}{2})
   \GR(s_{n-1}-t_{n-1}-\tfrac{\nu_n}{2}) \GR(s_{n}-t_{n-1}+\tfrac{\nu_n}{2}) \GR(s_{n}-t_{n-2}-\tfrac{\nu_n}{2})
\\
& \times 
   \GR(s+s_{n-2}-s_{n-1}-s_n+\mu_{n-1})  \GR(-s_{n-2}+s_{n-1}+s_n+\nu_n)
\\
& \times 
  \frac{\GR( (n-1)s-s_{n-1}-s_n + |\mu|-\mu_{n-1})}{ \GR( (n-1)s-s_{n-2} + |\mu| -\mu_{n-1}+\nu_n) }
\\
& \times  
 \frac{ \GR(ns-2s_{n-1}+|\mu|) \GR(ns -2s_n+|\mu|)}{ \GR(2ns-2s_{n-1}-2s_n+2|\mu|)}   
  \,\prod_{i=1}^n ds_i \prod_{i=1}^{n-1} dt_i \prod_{i=1}^{n-3} dq_i.
\end{align*}
Applying Proposition \ref{prop:Pieri_GLn}, with
$$ 
  s_i = is -s_{i}, \ \ \ 
   \sigma= s - \nu_n
$$
to the integral in each $q_i$ ($1\le i\le n-3$), we get
\begin{align*}
   I_{\nu, \mu}^n(s) 
& =  \prod_{i=1}^{n-2} \GR(s+\nu_n+\mu_i) \cdot \GR(s+\nu_n+\mu_n) \prod_{i=1}^{n-1} \GR(s-\nu_n+\mu_i) \\
& \times \frac{2^{n-3}}{(4\pi \sqrt{-1})^{2n-1}}  
    \int_{s_1,\dotsc,s_n, t_1,\dotsc, t_{n-1}}
      V_{n-1, \tilde{\nu}}^D(\bt)  
      V_{n-2, (\mu_1,\dotsc, \mu_{n-2})}^A(s-s_1,\dotsc,(n-3)s-s_{n-3})
\\
& \times   \GR(-s_{n-3}+s_{n-2}-\nu_n) 
    \frac{ \GR( (n-2)s-s_{n-2} + |\mu|-\mu_{n-1}-\mu_n)}{ \GR( (n-2)s-s_{n-3}+|\mu|-\mu_{n-1}-\mu_n-\nu_n)} 
\\
& \times      \prod_{i=1}^{n-3} \GR(s_i-t_i)   
  \prod_{i=1}^{n-3} \GR(s+s_{i-1}-t_i+\mu_{n})
\\
& \times      
    \GR(2s+s_{n-4}-t_{n-2}-t_{n-1}+\mu_{n-1}+\mu_n)
\\
& \times
\prod_{i=1}^{n-4} \frac{\GR(2s+s_{i-1}-t_{i+1}+\mu_{n-1}+\mu_n)}{ \GR(2s+s_{i-1}+s_{i} -t_i-t_{i+1} +\mu_{n-1}+\mu_n)   } 
\\
& \times 
 \frac{ \GR(s+s_{n-3}-s_{n-2}+\mu_{n})  }{  \GR(2s+s_{n-4} +s_{n-3}-t_{n-3}-t_{n-2}-t_{n-1}+\mu_{n-1}+\mu_n)}
  \\
& \times   
    \prod_{i=1}^{n-3} \GR(s+s_{i-1}-t_i+\mu_{n-1}) \cdot \GR(s+s_{n-3}-t_{n-2}-t_{n-1}+\mu_{n-1}) \\
& \times 
   \frac{  \GR(s_{n-2}-t_{n-2}-t_{n-1})   }
  {\GR(s+s_{n-2}-t_{n-2}-t_{n-1} + \mu_{n-1}-\nu_n) \GR(s_{n-1}+s_n-t_{n-2}-t_{n-1}+\nu_n) } 
\\
& \times 
   \GR(s_{n-1}-t_{n-2} + \tfrac{\nu_n}{2})
   \GR(s_{n-1}-t_{n-1}-\tfrac{\nu_n}{2}) \GR(s_{n}-t_{n-1}+\tfrac{\nu_n}{2}) \GR(s_{n}-t_{n-2}-\tfrac{\nu_n}{2})
\\
& \times 
   \GR(s+s_{n-2}-s_{n-1}-s_n+\mu_{n-1})  \GR(-s_{n-2}+s_{n-1}+s_n+\nu_n)
\\
& \times 
  \frac{\GR( (n-1)s-s_{n-1}-s_n + |\mu|-\mu_{n-1})}{ \GR( (n-1)s-s_{n-2} + |\mu| -\mu_{n-1}+\nu_n) }
\\
& \times  
 \frac{ \GR(ns-2s_{n-1}+|\mu|) \GR(ns -2s_n+|\mu|)}{ \GR(2ns-2s_{n-1}-2s_n+2|\mu|)}   
  \,\prod_{i=1}^n ds_i \prod_{i=1}^{n-1} dt_i.
\end{align*}

\noindent 
{\it Step 4}. 
We now treat the integrations over $s_{n-2},s_{n-1},s_n$.
For each integer $n\ge3$, define the following functions of the complex parameters $p,x,y,z,s,t_1,t_2,a,a_1,a_2,b$:
\begin{align*}
& D(x,y,z) \equiv D(p,x,y,z,s,t_1,t_2,a,a_1,a_2,b) \\
& = \frac{ \GR(-p+x-b)\GR( (n-2)s-x + a-a_1-a_2)}{ \GR( (n-2)s-p+a-a_1-a_2-b)} 
\\
& \times 
   \frac{  \GR(s+p-x+a_2)\GR(s+p-t_1-t_2+a_1)  \GR(x-t_1-t_2)   }
  {\GR(s+x-t_1-t_2 +a_1-b) \GR(y+z-t_1-t_2+b) } 
\\
& \times 
   \GR(y-t_1 + \tfrac{b}{2})\GR(y-t_2-\tfrac{b}{2}) \GR(z-t_2+\tfrac{b}{2}) \GR(z-t_1-\tfrac{b}{2})
\\
& \times 
   \GR(s+x-y-z+a_1)  \GR(-x+y+z+b)
\\
& \times 
  \frac{\GR( (n-1)s-y-z +a-a_1)}{ \GR( (n-1)s-x + a -a_1+b) } \cdot 
 \frac{ \GR(ns-2y+a) \GR(ns -2z+a)}{ \GR(2ns-2y-2z+2a)}   
\end{align*}
and
\begin{align*}
& E(p,y,z,s,t_1,t_2,a,a_1,a_2) \\
& = \frac{ \GR(s+p-t_1-t_2+a_2) \GR(s+p-y-z+a_1) \GR(s-p+y+z+a_2)}{ \GR((n-1)s-p+a-a_1)}
\\
& \times \frac{ \GR(y-t_1) \GR(z-t_2) \GR(s+y-t_2+a_2) \GR(s+z-t_1+a_2)}{ \GR(2s+y+z-t_1-t_2+2a_2)} 
\\
& \times \frac{ \GR( (n-1)s-2y+a-a_2) \GR((n-1)s-2z+a-a_2) }{\GR(2(n-1)s-2y-2z+2a-2a_2)}
\\
& \times \GR((n-2)s-y-z+a-a_1-a_2) .
\end{align*}

\begin{lem} \label{lem:edge_integral}
We have
\begin{align} \label{eqn:edge_integral}
\begin{split}
& \frac{1}{ (4\pi \sqrt{-1})^3} \int_{x,y,z} D(p,x,y,z,s,t_1,t_2,a,a_1,a_2,b)  \,dx\,dy\,dz \\
& = \frac{\GR(s+a_1+b) \GR(s+a_2-b)}{ \GR(2s+a_1+a_2) }
    \cdot \frac{1}{ (4\pi \sqrt{-1})^2} \int_{y,z} E(p,y,z,s,t_1,t_2,a,a_1,a_2) \,dy \,dz.
\end{split}
\end{align}
\end{lem}

\bpf
On the left hand side of (\ref{eqn:edge_integral}), we apply Lemma \ref{lem:int3} with
\begin{align*}
 (a,b,c,d) & = (s-y-z+a_1,  -p-b, -t_1-t_2,  s-t_1-t_2+a_1-b), \\
 (\alpha,\beta, \gamma,\delta) & = ((n-2)s+a-a_1-a_2, y+z+b,  s+p+a_2,  (n-1)s+a-a_1+b) ,
\end{align*}to the integral in $x$.  We then substitute
$x \to x + y+z $ to get 
\begin{align*}
& \frac{1}{ (4\pi \sqrt{-1})^3} \int_{x,y,z} D(x,y,z)  \,dx\,dy\,dz \\
& = \frac{ \GR(s+a_2-b) \GR( (n-2)s-t_1-t_2+a-a_1-a_2) }{ \GR(s+a_2+b) \GR( (n-2)s-p+a-a_1-a_2-b) }
\cdot \frac{1}{ (4\pi \sqrt{-1})^3} \int_{x,y,z} D_1(x,y,z) \,dx\, dy \,dz,
\end{align*}
where 
\begin{align*}
 D_1(x,y,z)
& = \frac{ \GR(2s+p-y-z+a_1+a_2) }{ \GR(y+z-t_1-t_2-b) }
\\
& \times \GR(y-t_1 + \tfrac{b}{2})\GR(y-t_2-\tfrac{b}{2}) \GR(z-t_2+\tfrac{b}{2}) \GR(z-t_1-\tfrac{b}{2}) \\
& \times \frac{ \GR( (n-2)s-x-y-z+a-a_1-a_2) \GR(-x+b) \GR(-x+p-t_1-t_2)}{ \GR( (n-2)s-x-t_1-t_2+a-a_1-a_2+b)} 
\\
& \times \frac{ \GR(x+s+a_1) \GR(x+y+z-p-b) \GR(x+s+a_2) }{ \GR(x+2s+a_1+a_2-b)}
\\
& \times \frac{ \GR(ns-2y+a) \GR(ns-2z+a)}{ \GR(2ns-2y-2z+2a)}.
\end{align*}

\medskip 

Let us consider the integral over $ y $. We replace $y\to -y$ and use Lemma \ref{lem:int4} with 
\begin{align*}
(a,b,c,d) & = ( (n-2)s-x-z+a-a_1-a_2,  2s+p-z+a_1+a_2,  ns +a,  2ns-2z+2a), \\
 (\alpha,\beta, \gamma,\delta) & = (-t_2-\tfrac{b}{2},  -p+x+z-b ,  -t_1+\tfrac{b}{2},   z-t_1-t_2-b),
\end{align*}
and substitute $ y \to x-y - \tfrac{b}{2} $. Then the left hand side of (\ref{eqn:edge_integral}) becomes 
\begin{align} \label{eqn:edge_integral_2}
\begin{split}
& \frac{ \GR(s+a_2-b) \GR( (n-2)s-t_1-t_2+a-a_1-a_2) }{ \GR(s+a_2+b) \GR( (n-2)s-p+a-a_1-a_2-b) }
\\
& \times  \frac{1}{ (4\pi \sqrt{-1})^4} \int_{x,y,z,w} D_2(x,y,z,w) \,dx\, dy\, dz\, dw,
\end{split} 
\end{align}
where 
\begin{align*}&D_2(x,y,z,w)\\
 &=  \frac{ \GR(x+s+a_1)\GR(x+s+a_2) \GR(-x+b)
   \GR(p-y-t_1) \GR(y-x-t_2)  }{ \GR(2s+x+a_1+a_2-b) \GR((n-2)s-x-t_1-t_2+a-a_1-a_2+b)\GR( ns-p+x+a) } 
\\
& \times \GR(w+2s-p+x+y+a_1+a_2+\tfrac{b}{2}) \GR(w+ 2s+x-t_2+a_1+a_2-\tfrac{b}{2})   \\
& \times \GR(-w+(n-2)s-x-t_1+a-a_1-a_2+\tfrac{b}{2}) \GR(-w+ (n-2)s-y+a-a_1-a_2-\tfrac{b}{2}) \\
& \times \frac{ \GR(z+w+2s-p+2x+a_1+a_2-b) \GR(z-p+y-\tfrac{b}{2}) \GR(z-t_2+\tfrac{b}{2}) }{ \GR(z+w+2s-p+y+x+a_1+a_2-t_2-b)}\\
& \times \frac{ \GR( -z+(n-2)s-x-t_1+a-a_1-a_2+\tfrac{b}{2}) \GR( -z+(n-2)s-y+a-a_1-a_2-\tfrac{b}{2} )
   }{ \GR( -z-w+ 2(n-2)s -y-x-t_1+2a-2a_1-2a_2)} 
\\
& \times \GR(-z-w-x+p).
\end{align*}
Here we substituted $ w \to -w-2s-x-a_1-a_2 $. 
Then we can apply Lemma \ref{lem:int4} in reverse, with 
\begin{align*}
 (a,b,c,d) & = (-p+y-\tfrac{b}{2},  2s-p+x+y+a_1+a_2+\tfrac{b}{2},   2s-p+2x+a_1+a_2-b,  \\
    & \qquad 4s+2x+2y-2p+2a_1+2a_2-b),  \\
(\alpha, \beta, \gamma, \delta) & = ( (n-2)s-x-t_1+a-a_1-a_2+\tfrac{b}{2},  (n-2)s-y+a-a_1-a_2-\tfrac{b}{2},    \\
& \qquad   p-x-t_2+\tfrac{b}{2},    2(n-2)s-x-y-t_1-t_2+2a-2a_1-2a_2+\tfrac{b}{2})
\end{align*}to the integral in $z$ and $w$, and follow with 
 the substitution $ z \to z-x+\frac{b}{2} $, to obtain 
\begin{align*}
 \frac{1}{ (4\pi \sqrt{-1})^4} \int_{x,y,z,w} D_2(x,y,z,w) \,dx \,dy \,dz\, dw = \frac{1}{ (4\pi \sqrt{-1})^3} \int_{x,y,z} D_3(x,y,z) \,dx \,dy \,dz 
\end{align*}
with 
\begin{align*}
 D_3(x,y,z) & = 
  \frac{ \GR(x+s+a_1)\GR(x+s+a_2) \GR(-x+b)}{ \GR( ns-p+x+a) } 
\\
& \times \GR( (n-2)s+x-y-z+a-a_1-a_2-b) \GR(-p-x+y+z) \\
& \times \frac{ \GR(p-y-t_1) \GR( (n-2)s-y-t_2+a-a_1-a_2) }
 { \GR(2(n-2)s-y-z-t_1-t_2+2a-2a_1-2a_2)}
\\
& \times \GR((n-2)s-z-t_1+a-a_1-a_2) \GR(p-z-t_2) \\
& \times \frac{\GR(2s-p+2y+a_1+a_2) \GR(2s-p+2z+a_1+a_2)}{ \GR(4s-2p+2y+2z+2a_1+2a_2)}
\\
& \times  \GR(2s-p+y+z+a_1+a_2+b).
\end{align*}

We can then apply Barnes' second lemma to the integral in $x$, to find that the expression in (\ref{eqn:edge_integral_2}) equals
\begin{align*}
&  \GR(s +a_1 +b) \GR(s+a_2-b) \cdot \frac{1}{ (4\pi \sqrt{-1})^2} \int_{y,z} D_4(y,z) \,dy \,dz,
\end{align*}
where
\begin{align*}
D_4(y,z) 
& = \frac{ \GR((n-2)s-t_1-t_2+a-a_1-a_2) \GR(p-y-t_1) \GR(p-z-t_2)}{ \GR(2(n-2)s-y-z-t_1-t_2+2a-2a_1-2a_2)}
\\
& \times \GR( (n-2)s-y-t_2+a-a_1-a_2) \GR( (n-2)s-z-t_1+a-a_1-a_2) \\
& \times \frac{ \GR(2s-p+2y+a_1+a_2) \GR(2s-p+2z+a_1+a_2)}{ \GR(4s-2p+2y+2z+2a_1+2a_2)} 
\\
& \times \frac{\GR((n-2)s-y-z+a-a_1-a_2) \GR(s-p+y+z+a_1)\GR(s-p+y+z+a_2) }{ 
       \GR( (n-1)s-p+a-a_1) \GR( (n-1)s-p+a-a_2) }.
\end{align*}

We apply Lemma \ref{lem:int4},  with 
\begin{align*}
 (a,b,c,d) & = (s-p+z+a_2,   s-p+z+a_1,  2s-p+a_1+a_2,  4s-2p+2z+2a_1+2a_2), \\
 (\alpha, \beta, \gamma, \delta) & = ( (n-2)s-z+a-a_1-a_2,  (n-2)s-t_2+a-a_1-a_2,  p-t_1, \\
 & \qquad  2(n-2)s-z-t_1-t_2+2a-2a_1-2a_2),
\end{align*}to the integral in $y$,
and then make the changes of variables $ z \to -z $ and $ w \to -w-s-a_1 $, to obtain
\begin{align*}
& \frac{1}{(4\pi \sqrt{-1})^2} \int_{y,z} D_4(y,z) \,dy dz\\
& = \frac{ 1}{ \GR(2s+a_1+a_2) \GR( (n-1)s-p+a-a_1) \GR((n-1)s-p+a-a_2)}  \\
& \times  \frac{1}{ (4\pi \sqrt{-1})^3} \int_{y,z,w} 
  D_5(y,z,w) \,dy\, dz \,dw,
\end{align*}
where 
\begin{align*}
 D_5(y,z,w)
& =  \GR(y-t_1) \GR( (n-2)s-y-t_2+a-a_1-a_2) 
\\
& \times \GR(w+s+p-y+a_1) \GR( w+(n-1)s-t_2+a-a_2) \\
& \times    \GR(-w+s-t_1+a_2) \GR(-w+s-p+y+a_2) 
   \GR( z+w +(n-1)s+a-a_2) \\
& \times 
\frac{   \GR( z +(n-2)s-y+a-a_1-a_2) \GR(z+p-t_2)   }{  \GR(z+w +(2n-3)s-y-t_2+2a-a_1-2a_2) }
\\
& \times \frac{ \GR(-z+s-t_1+a_2)   \GR(-z+s-p+y+a_2) \GR(-z-w-p)
} {\GR(-z-w+2s+y-p-t_1+2a_2)}.
   \end{align*}

We apply Lemma \ref{lem:int4} again in reverse for $ \int_{z,w} $ with 
\begin{align*}
(a,b,c,d) & =  ( (n-2)s-y+a-a_1-a_2,  s+p-y+a_1,  (n-1)s+a-a_2, \\
 & \qquad 2(n-1)s-2y+2a-2a_2), \\
(\alpha,\beta,\gamma,\delta) & = (s-t_1+a_2,  s-p+y+a_2,  -t_2,  2s+y-t_1-t_2+2a_2)
\end{align*}
and replace $ z $ by $-z $ to obtain
\begin{align*}
 &\frac{1}{ (4\pi \sqrt{-1})^2} \int_{y,z} D_4(y,z) \,dy\,dz 
\\&= \frac{1}{ \GR(2s+a_1+a_2)} \frac{1}{ (4\pi \sqrt{-1})^2} \int_{y,z} E(p,y,z,s,t_1,t_2,a,a_1,a_2) \,dy \,dz,
\end{align*}
as desired.
\epf

\bigskip 

Now we return to the integral $ I_{\nu,\mu}^n (s) $.
For $n=3$, the expression obtained in Step 2 can be written as follows, using the explicit formula for $V^A_2$:
\begin{align*}
 I^3_{\nu,\mu}(s)
 & =\GR(s-\nu_3+\mu_1)
     \GR(s-\nu_3+\mu_2)
     \GR(s+\nu_3+\mu_1)
     \GR(s+\nu_3+\mu_3)\\
 &\times\frac{1}{(4\pi\sqrt{-1})^5}
 \int_{x,y,z,t_1,t_2} V^D_{2,(\nu_1,\nu_2)}(t_1,t_2)\\
 &\times  D(0,x,y,z,s,t_1,t_2,|\mu|,\mu_2,\mu_3,\nu_3)
 \,dx \,dy \,dz \,dt_1\, dt_2.
\end{align*}
We apply 
Lemma \ref{lem:edge_integral} with
\[
(p,a,a_1,a_2,b)=(0,|\mu|,\mu_2,\mu_3,\nu_3).
\]
Moreover, the definitions of $E$, $M_2$, and $V^A_2$ give the identity
\begin{align*}
&E(0,y,z,s,t_1,t_2,|\mu|,\mu_2,\mu_3)\\
&=\frac{2}{\GR(2s+\mu_1+\mu_3)}
 V^A_{2,(\mu_1,\mu_2)}(s-y-z)
 M_2\bigl((y,z),(t_1,t_2);s+\mu_3\bigr)\\
&\times
 \frac{\GR(2s-2y+\mu_1+\mu_2) \GR(2s-2z+\mu_1+\mu_2)}{\GR(4s-2y-2z+2\mu_1+2\mu_2)}.
\end{align*}
Combining this identity with Lemma \ref{lem:edge_integral} and writing $(y,z)=(r_1,r_2)$, 
we obtain the $n=3$ case of the common $M_{n-1}$-integral displayed below.

\medskip 

For $ n \ge 4 $, 
we apply Lemma \ref{lem:edge_integral} with 
\[ 
  (p,x,y,z, t_1, t_2, a,a_1,a_2, b) 
 = (s_{n-3}, s_{n-2}, s_{n-1},s_n,  t_{n-2}, t_{n-1}, |\mu|,  \mu_{n-1},  \mu_n,  \nu_n)
\] 
to get 
\begin{align*}
   I_{\nu, \mu}^n(s) 
& =  \frac{1}{ \GR(2s+\mu_{n-1}+\mu_n)}  \prod_{i=1}^{n} \GR(s+\nu_n+\mu_i)  \prod_{i=1}^{n} \GR(s-\nu_n+\mu_i) \\
& \times \frac{2^{n-3}}{(4\pi \sqrt{-1})^{2n-2}}  
    \int_{s_1,\dotsc,s_{n-3},s_{n-1},s_n, t_1,\dotsc, t_{n-1}}
      V_{n-1, \tilde{\nu}}^D(\bt)  \\
& \times    V_{n-2, (\mu_1,\dotsc, \mu_{n-2})}^A(s-s_1,\dotsc,(n-3)s-s_{n-3})
\\
& \times      \prod_{i=1}^{n-3} \GR(s_i-t_i)   
  \prod_{i=1}^{n-3} \GR(s+s_{i-1}-t_i+\mu_{n})
    \prod_{i=1}^{n-3} \GR(s+s_{i-1}-t_i+\mu_{n-1}) \\
& \times
\prod_{i=1}^{n-4} \frac{ \GR(2s+s_{i-1}-t_{i+1}+\mu_{n-1}+\mu_n)}{ \GR(2s+s_{i-1}+s_{i} -t_i-t_{i+1} +\mu_{n-1}+\mu_n)   } 
\\
& \times \GR(s+s_{n-3}-t_{n-2}-t_{n-1}+\mu_n) \\
& \times 
 \frac{ \GR(2s+s_{n-4}-t_{n-2}-t_{n-1}+\mu_{n-1}+\mu_n)  }{  \GR(2s+s_{n-4} +s_{n-3}-t_{n-3}-t_{n-2}-t_{n-1}+\mu_{n-1}+\mu_n)}
  \\
& \times 
\frac{  \GR(s+s_{n-3}-s_{n-1}-s_n+\mu_{n-1}) \GR(s-s_{n-3}+s_{n-1}+s_n+\mu_n)}{ 
   \GR((n-1)s-s_{n-3}+|\mu|-\mu_{n-1})}
\\
& \times \frac{ \GR(s_{n-1}-t_{n-2}) \GR(s_n-t_{n-1}) \GR(s+s_{n-1}-t_{n-1}+\mu_n) \GR(s+s_n-t_{n-2}+\mu_n)}{ 
   \GR(2s+s_{n-1}+s_n-t_{n-2}-t_{n-1}+2\mu_n)} 
\\
& \times \frac{ \GR( (n-1)s-2s_{n-1}+|\mu|-\mu_n) \GR((n-1)s-2s_n+|\mu|-\mu_n) }{\GR(2(n-1)s-2s_{n-1}-2s_n+2|\mu|-2\mu_n)}
\\
& \times \GR((n-2)s-s_{n-1}-s_n+|\mu|-\mu_{n-1}-\mu_n)
   \prod_{i=1}^{n-3} ds_i ds_{n-1}ds_n \prod_{i=1}^{n-1} dt_i.
\end{align*}
Therefore we get 
\begin{align*}
I_{\nu, \mu}^n(s) 
& = \frac{1}{ \GR(2s+\mu_{n-1}+\mu_n)} \prod_{i=1}^{n} \GR(s+\nu_n+\mu_i)  \prod_{i=1}^{n} \GR(s-\nu_n+\mu_i) 
\\
& \times \frac{2^{n-3}}{(4\pi \sqrt{-1})^{n+1}}  
    \int_{s_{n-1}, s_{n}, t_1,\dotsc, t_{n-1} }
      V_{n-1, \tilde{\nu}}^D(\bt)  \cdot \GR(s-t_1+\mu_n) 
\\
& \times T_{n-2, (\mu_1,\dotsc,\mu_{n-2})}(t_1,\dotsc, t_{n-3}, t_{n-2}+t_{n-1};  s, s_{n-1}+s_n,  \mu_{n-1}, \mu_n) 
\\
& \times \frac{ \GR(s_{n-1}-t_{n-2}) \GR(s_n-t_{n-1}) \GR(s+s_{n-1}-t_{n-1}+\mu_n) \GR(s+s_n-t_{n-2}+\mu_n)}{ 
   \GR(2s+s_{n-1}+s_n-t_{n-2}-t_{n-1}+2\mu_n)} 
\\
& \times \frac{ \GR( (n-1)s-2s_{n-1}+|\mu|-\mu_n) \GR((n-1)s-2s_n+|\mu|-\mu_n) }{\GR(2(n-1)s-2s_{n-1}-2s_n+2|\mu|-2\mu_n)}
\\
& \times \GR((n-2)s-s_{n-1}-s_n+|\mu|-\mu_{n-1}-\mu_n)
  \, ds_{n-1}ds_n \prod_{i=1}^{n-1} dt_i. 
\end{align*}
By Proposition \ref{prop:new_GLn}, after the change of variables $ (s_{n-1},s_n) \to (r_{n-2},r_{n-1}) $, we find that 
\begin{align*}
I_{\nu, \mu}^n(s) 
& = \prod_{i=1}^{n-1} \frac{1}{ \GR(2s+\mu_i+\mu_n)} \prod_{i=1}^{n} \GR(s+\nu_n+\mu_i)  \prod_{i=1}^{n} \GR(s-\nu_n+\mu_i) 
\\
& \times \frac{2^{n-2}}{(4\pi \sqrt{-1})^{2n-2}}  
    \int_{t_1,\dotsc, t_{n-1}, r_1,\dotsc, r_{n-1} }
      V_{n-1, \tilde{\nu}}^D (\bt) \cdot \GR(s-t_1+\mu_n)\\
& \times V_{n-1, (\mu_1,\dotsc,\mu_{n-1})}^A (s-r_1,2s-r_2,\dotsc, (n-3)s-r_{n-3}, (n-2)s-r_{n-2}-r_{n-1})
\\
& \times \GR(s-t_{n-3}+r_{n-2}+r_{n-1}+\mu_n) \GR(2s- t_{n-2}-t_{n-1} +r_{n-2}+r_{n-1}+2\mu_n)
\\
& \times \prod_{i=1}^{n-3} \GR(r_i-t_i)  \prod_{i=1}^{n-4} \GR(r_i-t_{i+1}+s+\mu_n) \cdot \GR(r_{n-3}-t_{n-2}-t_{n-1}+s+\mu_n)
\\
& \times
\prod_{i=1}^{n-3} \GR(r_i-t_{i-1}+s+\mu_n) \prod_{i=2}^{n-3} \GR(r_i-t_i+2s+2\mu_n)
\\
& \times  \prod_{i=1}^{n-4} \frac{1}{ \GR(r_i+r_{i+1}-t_i-t_{i+1}+2s+2\mu_n)} 
\\
& \times \frac{ 1}{ \GR(r_{n-3}-t_{n-3}-t_{n-2}-t_{n-1} +2s+r_{n-2}+r_{n-1}+2\mu_n)} 
\\
& \times \frac{ \GR(r_{n-2}-t_{n-2}) \GR(r_{n-1}-t_{n-1}) \GR(s+r_{n-2}-t_{n-1}+\mu_n) \GR(s+r_{n-1}-t_{n-2}+\mu_n)}{ 
   \GR(2s+r_{n-2}+r_{n-1}-t_{n-2}-t_{n-1}+2\mu_n)} 
\\
& \times \frac{ \GR( (n-1)s-2r_{n-2}+|\mu|-\mu_n) \GR((n-1)s-2r_{n-1}+|\mu|-\mu_n) }{\GR(2(n-1)s-2r_{n-2}-2r_{n-1}+2|\mu|-2\mu_n)}
 \, \prod_{i=1}^{n-1} dt_i \prod_{i=1}^{n-1} dr_i.
\end{align*}
Thus, for  $n \ge 3 $, we have 
\begin{align*}
I_{\nu, \mu}^n(s) 
& = \prod_{i=1}^{n-1} \frac{1}{ \GR(2s+\mu_i+\mu_n)} \prod_{i=1}^{n} \GR(s+\nu_n+\mu_i)  \prod_{i=1}^{n} \GR(s-\nu_n+\mu_i) 
\\
& \times \frac{2^{n-2}}{(4\pi \sqrt{-1})^{2n-2}}  
    \int_{t_1,\dotsc, t_{n-1}, r_1,\dotsc, r_{n-1} }
      V_{n-1, \tilde{\nu}}^D(\bt) M_{n-1} (\br,  \bt; s+\mu_n) \\
&  \times V_{n-1, (\mu_1,\dotsc,\mu_{n-1})}^A (s-r_1,2s-r_2,\dotsc, (n-3)s-r_{n-3}, (n-2)s-r_{n-2}-r_{n-1})
\\
& \times \frac{ \GR( (n-1)s-2r_{n-2}+|\mu|-\mu_n) \GR((n-1)s-2r_{n-1}+|\mu|-\mu_n) }{\GR(2(n-1)s-2r_{n-2}-2r_{n-1}+2|\mu|-2\mu_n)}
 \, \prod_{i=1}^{n-1} dt_i \prod_{i=1}^{n-1} dr_i,
\end{align*}
where $ \br = (r_1,\dotsc,r_{n-1}). $
Applying Proposition \ref{prop:Pieri_SO2n} with $\sigma = s+\mu_n $  
and then using Proposition \ref{prop:I}, we obtain (\ref{eqn:SO_zeta_rec}) for all $n \ge 3 $. 
This completes the proof of Proposition \ref{prop:SOGL_induction}(ii), and hence that of Theorem \ref{thm:SOGL}.
\epf

\section{Archimedean zeta integrals on $ {\rm Sp}_{2n} \times {\rm GL}_n $} 
\label{sec:Arch_C}
 
We explain how the integral considered in this section arises from
Method C of \cite[Part B, Section 4 and Section 5.3]{GePS}.
In this construction the unfolding of the theta series uses the two
${\rm GL}_n$-orbits in the space of row vectors, namely $\{0\}$ and
$e_n {\rm GL}_n$, where $e_n=(0,\ldots,0,1)$.
The contribution of the zero orbit vanishes by cuspidality, and the nonzero orbit gives the factor
$(\omega(g)\Phi)(e_n)$; see \cite[Part B, Lemma 4.3.2 and Theorem 4.C]{GePS}.
We specify the real local data, including the metaplectic and compact
characters, to obtain precisely the integral used below.

Put $G={\rm Sp}_{2n}(\R)$, with the realization fixed in Section \ref{sec:Wh},
and $K=G\cap {\rm O}(2n)$.
Let $P=M_PN_P$ be the Siegel parabolic subgroup of $G$, where
\begin{align*}
 M_P & = 
\{ 
 m(b)=
 \begin{pmatrix}
 b&0\\
 0&J_n\,{}^t b^{-1}J_n
 \end{pmatrix} \mid
 b\in {\rm GL}_n(\R) \},
\\
 N_P 
& = \{
 n(S)=
 \begin{pmatrix}
 I_n&SJ_n\\
 0&I_n
 \end{pmatrix} \mid S={}^tS\in M_n(\R)\}.
\end{align*}
We know 
$
 m(b)n(S)m(b)^{-1}=n(bS\,{}^t b),
$
and 
$ \delta_P(m(b))=|\det b \,|^{n+1}. $

Write $U_C$ and $ U_A $ for the upper triangular maximal unipotent subgroups of $ G$ and ${\rm GL}_n(\R)$, respectively,
and $ \psi_C $ and $ \psi_A $ for the characters defined in Section \ref{sec:Wh}.
Every $u \in U_C$ has a unique expression
\[
 u=m(z)n(S),\qquad z\in U_A,\quad S={}^tS.
\]
In these coordinates, we have
\begin{align}
 \label{eq:SpGL-whittaker-characters}
 \psi_C(m(z)n(S))
 =\psi_A(z)\exp(2\pi  \sqrt{-1} S_{nn}).
\end{align}
Indeed, the last row of $z$ is $e_n$, so that
$u_{n,n+1}=S_{nn}$.

\medskip
\noindent
\textit{The Weil representation and the Gaussian.}
We use the circle central extension employed in
\cite[Part B, (4.2.1)]{GePS},
\[
 1\longrightarrow\mathbb T\longrightarrow\widetilde G
 \mathop{\longrightarrow}^{\operatorname{pr}}G
 \longrightarrow1,
 \qquad
 \mathbb T=\{\lambda\in\mathbb C \mid |\lambda|=1\}.
\]
Let $\omega_-$ be its Schr\"odinger representation on
$\mathcal S(\R^n)$, with the inverse additive character and
the following normalization of the distinguished lifts of $M_P$ and $N_P$:
\begin{align}
 \label{eq:SpGL-Weil-unipotent}
 (\omega_-(\widetilde n(S))\Phi)(x)
 &=\exp(-2\pi \sqrt{-1} xS\,{}^t \! x)\Phi(x),\\
 \label{eq:SpGL-Weil-Levi}
 (\omega_-(\widetilde m(b))\Phi)(x)
 &=|\det b\,|^{1/2}\Phi(xb),\\
 \label{eq:SpGL-Weil-center}
 \omega_-(\lambda)\Phi&=\lambda\Phi,
 \qquad \lambda\in\mathbb T.
\end{align}
In particular, the quadratic expression in \eqref{eq:SpGL-Weil-unipotent} contains no factor $1/2$.
Here $ \Phi \in \mathcal{S}(\R^n) $ and $x$ is a row vector.
These are the conventions of  \cite[Part B, (4.2.2), (4.2.3)]{GePS},
with the additive character replaced by its inverse.

Put $\widetilde K=\operatorname{pr}^{-1}(K)$. 
To obtain a right $K$-invariant integrand, 
we choose a Gaussian spanning a one-dimensional $\widetilde K$-type and 
take the inducing section in the inverse type. 
With the quadratic-phase convention \eqref{eq:SpGL-Weil-unipotent}, the appropriate Gaussian, 
normalized to have value $1$ at the origin, is
\begin{align}
 \label{eq:SpGL-Gaussian}
 \Phi_0(x)=
 \exp\Bigl(-2\pi\sum_{i=1}^n x_i^2\Bigr),
 \qquad \Phi_0(0)=1.
\end{align}
For completeness, the infinitesimal action of the rotation generator in the $i$-th symplectic coordinate plane is
\[
 \frac{\sqrt{-1} }{8\pi}\frac{\partial^2}{\partial x_i^2} -2\pi \sqrt{-1} x_i^2.
\]
For $ c>0$, the coefficient of $ x_i^2 \exp( -c \sum_{l} x_l^2) $ in the action of this generator is 
\[
  \sqrt{-1} \Bigl( \frac{c^2}{2\pi} -2\pi \Bigr),
\]
so the eigenvector condition forces $c=2\pi$.
Together with invariance under orthogonal changes of the
$x$-coordinates, this proves the assertion about $\widetilde K$.
Denote the resulting unitary character of $\widetilde{K}$ by $\kappa$:
\begin{align}
 \label{eq:SpGL-Gaussian-K-type}
 \omega_-(\widetilde k)\Phi_0
 =\kappa(\widetilde k)\Phi_0,
 \qquad \widetilde k\in\widetilde K.
\end{align}
In particular,
\[
 \kappa(\lambda)=\lambda\quad(\lambda\in\mathbb T),
 \qquad
 \kappa(\widetilde m(q))=1\quad(q\in {\rm O}(n)).
\]

\medskip
\noindent
\textit{The inducing section.}
Let $\pi_\mu$ be the spherical principal series of ${\rm GL}_n(\R)$ with parameter $\mu$.
We realize it in its $\psi_A^{-1}$-Whittaker model as in Section \ref{sec:Arch_D}:
\[
 \epsilon_n=
 \operatorname{diag}
 \bigl((-1)^{n-1},(-1)^{n-2},\ldots,-1,1\bigr),
 \qquad
 W^{A,-}_{n,\mu}(b)=W^A_{n,\mu}(\epsilon_n b), \quad b \in {\rm GL}_n(\R).
\]
Indeed, $W^{A,-}_{n,\mu}$ is a right ${\rm O}(n)$-invariant $ \psi_A^{-1} $-Whittaker function satisfying
\begin{align*}
 W^{A,-}_{n,\mu}(\alpha[a_1,\ldots,a_n])
 &=W^A_{n,\mu}(\alpha[a_1,\ldots,a_n]).
\end{align*}

Let $\widetilde P=\operatorname{pr}^{-1}(P)$. On $\widetilde{P} $, take the inducing representation defined by 
\[
 \lambda\widetilde m(b)\widetilde n(S)
 \longmapsto
 \lambda^{-1}|\det b\,|^{s-1/2}\pi_\mu(b).
\]
Induction from $\widetilde P$ to $\widetilde G$ is normalized.
The two metaplectic factors are thus taken with mutually inverse
central characters.
Let $F_{\mu,s}$ be the scalar-valued Whittaker realization of the
section of right $\widetilde K$-type $\kappa^{-1}$ defined by
\begin{align}
 \label{eq:SpGL-section}
 F_{\mu,s}
 \bigl(\widetilde m(b)\widetilde n(S)\widetilde k\bigr)
 =|\det b\,|^{s+n/2}
 W^{A,-}_{n,\mu}(b)\kappa(\widetilde k)^{-1}.
\end{align}

Formula \eqref{eq:SpGL-section} is well-defined: the ambiguity in
the $PK$-decomposition comes from $M_P\cap K=m({\rm O}(n))$, on which
$W^{A,-}_{n,\mu}$ is right invariant and the distinguished lifts
have $\kappa=1$.
It follows that
\begin{align}
 \label{eq:SpGL-section-center}
 F_{\mu,s}(\lambda\widetilde g)
 &=\lambda^{-1}F_{\mu,s}(\widetilde g),\\
 \label{eq:SpGL-section-K-type}
 F_{\mu,s}(\widetilde g\widetilde k)
 &=\kappa(\widetilde k)^{-1}F_{\mu,s}(\widetilde g).
\end{align}
Thus $F_{\mu,s}$ is not being taken to be $\widetilde K$-fixed;
its compact character cancels that of the Gaussian.

\medskip
\noindent
\textit{The local integral and its radial expression.}
In the above conventions, the real local integral supplied by Method C is
\begin{align}
 \label{eq:SpGL-local-zeta}
 Z(s,W^C_{n,\nu},F_{\mu,s},\Phi_0)
 =\int_{U_C\backslash G}
 W^C_{n,\nu}(g)
 F_{\mu,s}(\widetilde g)
 (\omega_-(\widetilde g)\Phi_0)(e_n)\,dg,
\end{align}
where $\widetilde g$ is any lift of $g$.
This is the local form of
\cite[Part B, Theorem 4.C and Section 5.3]{GePS}, with the opposite
Whittaker characters in the two factors paired with $W^C_{n,\nu}$.
The product in \eqref{eq:SpGL-local-zeta} is independent of the
choice of lift by \eqref{eq:SpGL-Weil-center} and \eqref{eq:SpGL-section-center}.

We check left $U_C$-invariance explicitly.
For $u=m(z)n(S)$, choose the lift $ \widetilde{u} = \widetilde{m}(z) \widetilde{n}(S) $. 
Since $e_nz=e_n$, equations 
\eqref{eq:SpGL-Weil-unipotent} and \eqref{eq:SpGL-Weil-Levi} give
\[
 (\omega_-(\widetilde u\widetilde g)\Phi_0)(e_n)
 =\exp(-2\pi \sqrt{-1} S_{nn})
 (\omega_-(\widetilde g)\Phi_0)(e_n),
\]
whereas
\[
 F_{\mu,s}(\widetilde u\widetilde g)
 =\psi_A(z)^{-1}F_{\mu,s}(\widetilde g).
\]
These two factors cancel the character
$\psi_C(u)$ of $W^C_{n,\nu}$ by
\eqref{eq:SpGL-whittaker-characters}.
Moreover,
\eqref{eq:SpGL-Gaussian-K-type} and
\eqref{eq:SpGL-section-K-type}, together with the right
$K$-invariance of $W^C_{n,\nu}$, show that the product is right
$K$-invariant.

Write
\[
 a=\beta[a_1,\ldots,a_n]=m(b),
 \qquad b=\alpha[a_1,\ldots,a_n].
\]
Normalize the quotient measure by its Iwasawa expression
\[
 dg=
 \prod_{i=1}^n a_i^{-2(n+1-i)} \frac{da_i}{a_i}\,dk,
 \qquad
 \int_Kdk=1.
\]
For the distinguished lift of $a$, we have
\begin{align}
 \label{eq:SpGL-Gaussian-on-torus}
 (\omega_-(\widetilde a)\Phi_0)(e_n)
 &=(\det b)^{1/2}\Phi_0(e_nb)
   =(\det b)^{1/2}\exp(-2\pi a_n^2),\\
 \label{eq:SpGL-section-on-torus}
 F_{\mu,s}(\widetilde a)
 &=(\det b)^{s+n/2}W^A_{n,\mu}(b).
\end{align}
Consequently,
\begin{align*}
 Z(s,W^C_{n,\nu},F_{\mu,s},\Phi_0)
 & = 
 \int_{(\R_+)^n}
 W^C_{n,\nu}(a)W^A_{n,\mu}(b)
 \\
&\times
 (a_1\cdots a_n)^{s+(n+1)/2}\exp(-2\pi a_n^2) 
 \prod_{i=1}^n a_i^{-2(n+1-i)} \frac{da_i}{a_i}.
\end{align*}
It follows that
\begin{align}
\begin{split}
 \label{eq:SpGL-definition-J}
 Z(s,W^C_{n,\nu},F_{\mu,s},\Phi_0)
 & = 
 \int_{(\R_+)^n}
 \widetilde W^C_{n,\nu}(\beta[a_1,\ldots,a_n])
\\
 &\times
 \widetilde W^A_{n,\mu}(\alpha[a_1,\ldots,a_n])
 \cdot (a_1\cdots a_n)^s
 \exp(-2\pi a_n^2)
 \prod_{i=1}^n\frac{da_i}{a_i}.
\end{split}
\end{align}
We denote this integral by $ J_{\nu,\mu}^n(s) $.
Thus the integral $J^n_{\nu,\mu}(s)$ is exactly the real local
Rankin-Selberg integral for the data and measure specified above.


Using Theorem \ref{thm:SOGL}, we can prove the following.

\begin{thm}   \label{thm:SpGL}
For $ n \ge 2 $, $ \nu=(\nu_1,\dotsc,\nu_n) \in \C^n $ and $ \mu = (\mu_1,\dotsc,\mu_n) \in \C^n $, we have 
\begin{align*}
 J_{\nu,\mu}^n(s)  =  \prod_{1 \le i \le j \le n} \frac{1}{\GR(2s+\mu_i+\mu_j)}
  \prod_{i=1}^n \GR(s+\mu_i)  \prod_{1 \le i, j \le n} \{ \GR(s+\nu_i+\mu_j) \GR(s-\nu_i+\mu_j) \}.
\end{align*}
\end{thm}

\bpf
Inserting the Mellin-Barnes integral representation of $ \widetilde{W}_{n,\mu}^A $, we find that
\begin{align*}
 J_{\nu,\mu}^n(s) 
& = \frac{1}{ (2\pi \sqrt{-1})^{n-1} } \int_{s_1,\dotsc, s_{n-1} } V_{n, \mu}^A(s_1,\dotsc,s_{n-1}) 
\\
& \times 
  \int_{(\R_+)^n} \widetilde{W}_{n, \nu }^C (\beta[a_1,\dotsc,a_n]) \exp(-2\pi a_n^2)
\\
& \times  
   \,  a_1^{s-s_1} a_2^{s+s_1-s_2} \cdots a_{n-1}^{s+s_{n-2}-s_{n-1}} a_n^{s+s_{n-1}+|\mu|}  \prod_{i=1}^n \frac{da_i}{a_i} \prod_{i=1}^{n-1}  ds_i.
\end{align*}
Since\[
 \exp(-2\pi a_n^2) = \frac{1}{4\pi \sqrt{-1}} \int_q \GC(q) \,a_n^{-2q} \,dq,
\]
we get 
\begin{align*}
 J_{\nu,\mu}^n(s) 
& = \frac{2^{-1} }{ (2\pi \sqrt{-1})^n} \int_{s_1,\dotsc,s_{n-1},q} \GC(q) V_{n, \mu}^A (s_1,\dotsc,s_{n-1}) 
\\
& \times \int_{(\R_+)^n} 
   \widetilde{W}_{n, \nu}^C  ( \beta[a_1,\dotsc, a_n]) 
 \prod_{i=1}^{n-1} \Bigl( \frac{a_i}{a_{i+1} } \Bigr)^{\! is-s_i} \cdot
  (a_n^2)^{\frac{ns+|\mu|}{2}-q} \prod_{i=1}^{n } \frac{da_i}{a_i} \prod_{i=1}^{n-1} ds_i \, dq.
\end{align*}
Then we have
\begin{align*}
 J_{\nu,\mu}^n(s) 
& = \frac{2^{-2} }{ (2\pi \sqrt{-1})^n} \int_{s_1,\dotsc, s_{n-1},q}   
   V_{n,\nu}^C (s-s_1, 2s- s_2,\dotsc, (n-1)s-s_{n-1}, \tfrac{ns+|\mu|}{2}-q) 
\\
& \times V_{n, \mu }^A(s_1,s_2, \dotsc,s_{n-1}) \GC(q) 
 \prod_{i=1}^{n-1} ds_i  \,dq.
\end{align*}
In view of Theorem \ref{thm:SOGL}, it is enough to show  that
\begin{align} \label{eqn:zeta_SO_Sp}
    J_{\nu,\mu}^n(s) =  \prod_{i=1}^n \GR(s+\mu_i) \cdot I_{\nu,\mu}^n(s).
\end{align}
To this end we have, by Proposition \ref{prop:exp_Sp2n}, 
\begin{align*}
 J_{\nu,\mu}^n (s) 
& = \frac{2^{-2} }{ (2\pi \sqrt{-1})^n (4\pi \sqrt{-1})^{n} } \int_{s_1, \dotsc, s_{n-1},q, t_1,\dotsc, t_n} \GC(q) 
  V_{n,\nu}^D(t_1,\dotsc,t_n)  \\
& \times  \prod_{i=1}^{n-2} \GR(is-s_i-t_i)  \prod_{i=1}^{n-2} \GR(is-s_i-t_{i-1})  \\
& \times  \GR((n-1)s-s_{n-1}-t_{n-2})  \GR((n-1)s-s_{n-1}-t_{n-1}-t_n)  \\
& \times \GR(\tfrac{ns+|\mu|}{2}-q-t_{n-1})
    \GR(\tfrac{ns+|\mu|}{2}-q-t_n)
\\
& \times V_{n, \mu }^A(s_1,\dotsc,s_{n-1}) 
   \, \prod_{i=1}^{n-1} ds_i \prod_{i=1}^ndt_i dq.
\end{align*}
The integral in $q$ may then be evaluated by way of  Barnes' first lemma and the fact that $\GC(q) = \GR(q) \GR(q+1) $, which together tell us that \[
 \frac{1}{4 \pi \sqrt{-1}} \int_q
 \GC(q) \GR(x-q) \GR(y-q) \,dq = \frac{ \GC(x) \GC(y)}{ \GR(x+y+1) } = \frac{2 \GR(2x) \GR(2y) \GR(x+y)}{\GR(2x+2y)}.
\]
Further, the integrations in $s_1, s_2,\ldots, s_{n-1}$ may be performed using Proposition \ref{prop:Pieri_GLn},  
with 
\[ 
  \sigma = s,  \qquad  s_i = is- t_i  \ (1 \le i \le n-2), \qquad s_{n-1} = (n-1)s-t_{n-1}-t_n. 
\]
The result is 
\begin{align*}
 J_{\nu,\mu}^n(s) 
& =  \prod_{i=1}^n \GR(s+\mu_i) \cdot  \frac{2^{-1} }{ (2\pi \sqrt{-1})^{n} } \int_{t_1,\dotsc, t_n}  
  V_{n,\nu}^D(t_1,\dotsc,t_n)  \\
& \times V_{n, \mu }^A(s-t_1,2s-t_2,\dotsc,(n-2)s-t_{n-2}, (n-1)s-t_{n-1}-t_n) \\
& \times  
   \frac{ \GR(ns+|\mu|-2t_{n-1}) \GR(ns+|\mu|-2t_n) }{ \GR(2ns +2|\mu|-2t_{n-1}-2t_n)}
   \, \prod_{i=1}^{n}dt_i
\\
& = \prod_{i=1}^n \GR(s+\mu_i) \cdot I_{\nu,\mu}^n(s).
\end{align*}
Therefore we get (\ref{eqn:zeta_SO_Sp}) to complete the proof.
\epf

\section*{Statement of AI use}

The authors acknowledge that the idea for the calculation in the case $ {\rm SO}_{3,3} \times {\rm GL}_3 $ 
originated from a suggestion by ChatGPT (OpenAI). 
The authors carried out and verified the calculation in this case and developed and verified its extension to general $n$. 
The authors take full responsibility for the content of the paper.


\end{document}